\documentclass[12pt,a4paper]{amsart}
\usepackage{amssymb, enumerate, mdwlist, amsthm, amsmath, bm, graphicx, epsfig}
\usepackage{xcolor,hyperref} 
\usepackage[foot]{amsaddr}
\makeatletter
  \newcommand{\subsubsubsection}{\@startsection{paragraph}{4}{\z@}%
    {1.0\Cvs \@plus.5\Cdp \@minus.2\Cdp}%
    {.1\Cvs \@plus.3\Cdp}%
    {\reset@font\it\normalsize}
  }
  \makeatother
\hypersetup{
    colorlinks=true,
    citecolor=blue,
    linkcolor=red,
    urlcolor=orange,
}
\newtheorem{thm}{Theorem}
\newtheorem{prop}{Proposition}
\newtheorem{conj}{Conjecture}
\newtheorem{lemma}{Lemma}
\newtheorem{claim}{Claim}
\newtheorem{fact}{Fact}
\newtheorem{prob}{Problem}
\def\erf{\mathop{\mathrm{erf}}}
\def\tp{{\mathsf T}}
\def\N{\mathbb N}
\def\Z{\mathbb Z}
\def\R{\mathbb R}
\def\P{\mathbb P}
\def\AA{\mathcal A}
\def\BB{\mathcal B}
\def\FF{\mathcal F}
\def\GG{\mathcal G}
\def\AK{{\rm AK}}
\def\BD{{\rm BD}}
\def\bd{{\rm bd}}
\def\aa{{\alpha}}
\def\bb{{\beta}}
\title{The maximum measure of non-trivial 3-wise intersecting families}
\author{Norihide Tokushige}\thanks{The author was supported by JSPS KAKENHI 18K03399 and 23K032101.}
\address{Norihide Tokushige\\
University of the Ryukyus\\
1 Senbaru, Okinawa 903-0213 JAPAN}
\email{hide@edu.u-ryukyu.ac.jp}
\begin{document}
\maketitle
\begin{abstract}
Let $\GG$ be a family of subsets of an $n$-element set. 
The family $\GG$ is called non-trivial $3$-wise intersecting
if the intersection of any three subsets in $\GG$ is non-empty, but
the intersection of all subsets is empty.
For a real number $p\in(0,1)$ we define the measure of the family 
by the sum of $p^{|G|}(1-p)^{n-|G|}$ over all $G\in\GG$. 
We determine the maximum measure of non-trivial $3$-wise intersecting 
families. We also discuss the uniqueness and stability of the corresponding
optimal structure. These results are obtained by solving linear programming
problems.
\end{abstract}

\section{Introduction}
We determine the maximum measure of non-trivial $3$-wise intersecting
families, and discuss the stability of the optimal structure.
To make the statement precise let us start with some definitions.

Let $n\geq t\geq 1$ and $r\geq 2$ be integers.
For a finite set $X$ let $2^X$ denote the power set of $X$.
We say that a family of subsets $\GG\subset 2^{X}$ is
$r$-wise $t$-intersecting if $|G_1\cap\cdots\cap G_r|\geq t$ for all
$G_1,\ldots,G_r\in\GG$. 
If $t=1$ then we omit $t$ and say an $r$-wise 
intersecting family to mean an $r$-wise $1$-intersecting family.

Let $0<p<1$ be a real number and let $q=1-p$. 
For $\GG\subset 2^{X}$ we define its measure (or $p$-measure) 
$\mu_p(\GG:X)$ by
\[
 \mu_p(\GG:X):=\sum_{G\in\GG}p^{|G|} q^{|X|-|G|}.
\]
We mainly consider the case $X=[n]$, where $[n]:=\{1,2,\ldots,n\}$.
In this case we just write $\mu_p(\GG)$ to mean $\mu_p(\GG:[n])$.

We say that an $r$-wise $t$-intersecting family $\GG\subset 2^{[n]}$ is
non-trivial if $|\bigcap\GG|<t$, where $\bigcap\GG:=\bigcap_{G\in\GG}G$.
Let us denote the maximum $p$-measure of such families by $M_r^t(n,p)$, 
that is,
\[
 M_r^t(n,p):=\max\{\mu_p(\GG):\text{$\GG\subset 2^{[n]}$ is non-trivial
$r$-wise $t$-intersecting}\}.
\]
If a family $\GG\subset 2^{[n]}$ is non-trivial $r$-wise $t$-intersecting,
then so is 
\[
\GG':=\GG\sqcup\{G\sqcup\{n+1\}:G\in\GG\}\subset 2^{[n+1]}. 
\]
Since $\GG$ and $\GG'$ have the same $p$-measure, the function
$M_r^t(n,p)$ is non-decreasing in $n$ for fixed $r,t,p$, and we can define
\[
M_r^t(p):=\lim_{n\to\infty}M_r^t(n,p). 
\]
For simplicity if $t=1$ then we just write $M_r(n,p)$ and $M_r(p)$.

What is generally known about $M_r(n,p)$ and $M_r(p)$?
For the case $r=2$ we have the following.
\begin{align*}
 M_2(p)=\begin{cases}
       p & \text{if } 0<p\leq\frac12,\\
       1 & \text{if } \frac12<p<1.
      \end{cases}
\end{align*}
Indeed it is easy to see that $M_2(n,\frac12)=\frac12$, and
it is known from \cite{AKa} that $M_2(n,p)<p$ for $p<\frac12$.
Thus $M_2(p)\leq p$ for $p\leq\frac12$.
On the other hand we construct a non-trivial $r$-wise intersecting family by
$\FF:=(\{F\in 2^{[n]}:1\in F\}\setminus\{\{1\}\})\cup\{[2,n]\}$,
where $[i,j]:=[j]\setminus[i-1]$.
Then we have $\mu_p(\FF)=p-pq^{n-1}+qp^{n-1}\to p$ as $n\to\infty$.
Thus $M_2(p)=p$ for $p\leq\frac12$.
For the case $p>\frac12$ we construct a non-trivial $r$-wise intersecting 
family $\GG:=\{F\in 2^{[n]}:|F|>n/2\}$.
Then $\mu_p(\GG)=\sum_{k>n/2}\binom nkp^kq^{n-k}\to 1$ as $n\to\infty$,
and so $M_2(p)=1$ for $p>\frac12$.

The case $p=\frac12$ (and arbitrary $r\geq 2$) is also known.
Brace and Daykin \cite{BD} determined the maximum size of non-trivial 
$r$-wise intersecting families. In other words, they determined
$M_r(n,\frac12)$. To state their results we define a non-trivial
$r$-wise intersecting family $\BD_r(n)$ by
\[
 \BD_r(n):=\{F\in 2^{[n]}:|F\cap[r+1]|\geq r\}.
\]
Then $\mu_p(\BD_r(n))=(r+1)p^rq+p^{r+1}$. 

\begin{thm}[Brace and Daykin \cite{BD}]\label{thm:BD}
For $r\geq 2$ we have $M_r(n,\frac12)=\mu_{\frac12}(\BD_r(n))$.
If $r\geq 3$ then $\BD_r(n)$ is the only optimal family (up to isomorphism)
whose measure attains $M_r(n,\frac12)$.
\end{thm}
\noindent
Here two families $\FF,\GG\subset 2^{[n]}$ are isomorphic if there is 
a permutation $\tau$ on $[n]$ such that 
$\FF=\{\{\tau(g):g\in G\}:G\in\GG\}$. In this case we write $\FF\cong\GG$. 

The other thing we know is about the case $p$ close to $\frac12$. 
In this case we can extend Theorem~\ref{thm:BD} if $r\geq 8$ as follows.
\begin{thm}[\cite{T2008}]
Let $r\geq 8$. Then there exists $\epsilon=\epsilon(r)>0$ such that
$M_r(n,p)=\mu_p(\BD_r(n))$ for $|p-\frac12|<\epsilon$, and 
$\BD_r(n)$ is the only optimal family (up to isomorphism).
\end{thm}
\noindent
In \cite{T2008} it is conjectured the same holds for $r=6$ and $7$ as well.
On the other hand there is a construction showing that
if $r\leq 5$ then $M_r(n,p)>\mu_p(\BD_r(n))$ for $p$ close to $1/2$.

In this paper we focus on the case $r=3$.
We determine $M_3(p)$ for all $p$.

\begin{thm}\label{thm2}
For non-trivial $3$-wise intersecting families we have
\begin{align*}
 M_3(p)=\begin{cases}
       p^2 & \text{if } p\leq\frac13,\\
       4p^3q+p^4 & \text{if } \frac13\leq p\leq\frac12,\\
       p & \text{if } \frac12<p\leq\frac23,\\
       1 & \text{if } \frac23<p<1.
      \end{cases}
\end{align*}
\end{thm}
\begin{figure}
\includegraphics[width=10cm]{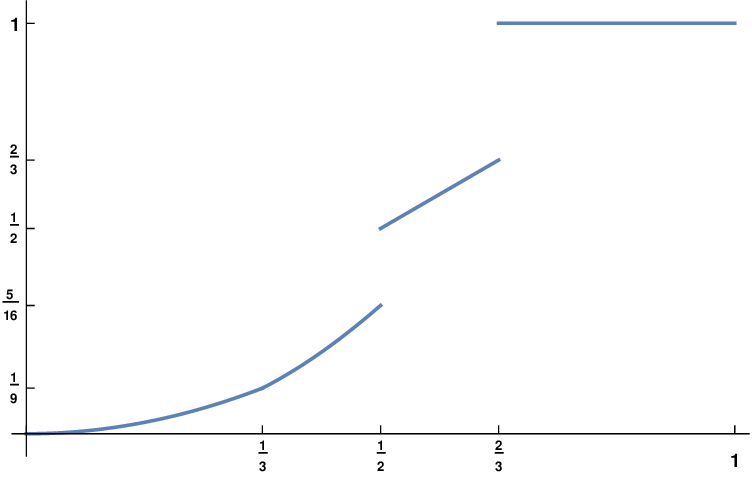}  
\caption{The graph of $M_3(p)$}\label{fig7}
\end{figure}
\noindent
In case $M_2(p)$ there is a jump at $p=\frac12$. In case $M_3(p)$ 
there are two jumps at $p=\frac12$ and $p=\frac 23$ as in 
Figure~\ref{fig7}, and $M_3(p)$ is continuous at $p=\frac13$ 
but not differentiable at this point. We also note that
$\mu_p(\BD_3(n))=4p^3q+p^4$. 

The most interesting part is the case $\frac13\leq p\leq\frac12$.
In this case we determine $M_3(n,p)$ and the corresponding optimal
structure.
\begin{thm}\label{thm:M3(n,p)}
Let $\frac13\leq p\leq\frac12$. Then we have $M_3(n,p)=\mu_p(\BD_3(n))$. 
Moreover, $\BD_3(n)$ is the only optimal family (up to isomorphism), 
that is, if $\FF\subset 2^{[n]}$ is a non-trivial $3$-wise intersecting
family with $\mu_p(\FF)=M_3(n,p)$ then $\FF\cong\BD_3(n)$.
\end{thm}

We also consider the stability of the optimal family for 
$\frac13\leq p\leq\frac12$. Roughly speaking
we will claim that if a non-trivial 3-wise intersecting family has measure
close to $M_3(n,p)$ then the family is close to $\BD_3(n)$ in structure.
A similar result is known for 2-wise $t$-intersecting families.
For comparison with our case let $\frac13<p<\frac25$ and $t=2$.
Note that $\BD_3(n)$ is a 2-wise 2-intersecting family.
If $\FF\subset 2^{[n]}$ is a 2-wise 2-intersecting family, then 
it follows from the Ahlswede--Khachatrian theorem (see Theorem~\ref{thm:AK})
that $\mu_p(\FF)\leq\mu_p(\BD_3(n))$.
Moreover, if $\mu_p(\FF)$ is close to $\mu_p(\BD_3(n))$ then $\FF$
is close to $\BD_3(n)$. This follows from a stability result (corresponding
to Theorem~\ref{thm:AK}) proved by Ellis, Keller, and Lifshitz.
Here we include a version due to Filmus applied to the case
$\frac13<p<\frac25$ and $t=2$.
\begin{thm}[Ellis--Keller--Lifshitz \cite{EKL}, 
Filmus \cite{Filmus}]\label{thm:EKL}
Let $\frac13<p<\frac25$. There is a constant $\epsilon_0=\epsilon_0(p)$ 
such that the following holds. If $\FF\subset 2^{[n]}$ is a $2$-wise 
$2$-intersecting family with $\mu_p(\FF)=\mu_p(\BD_3(n))-\epsilon$, 
where $\epsilon<\epsilon_0$, then there is a family $\GG\cong\BD_3(n)$ 
such that $\mu_p(\FF\triangle\GG)=O(\epsilon)$,
where the hidden constant depends on $p$ only.
\end{thm}
\noindent
We note that the condition $\mu_p(\FF\triangle\GG)=O(\epsilon)$ 
in Theorem~\ref{thm:EKL} cannot be replaced with the condition 
$\FF\subset\GG$. To see this, consider a 2-wise 2-intersecting family
\begin{align}\label{2w2i not 3w1i}
 \FF=\left(\BD_3(n)\setminus\{\{1,3,4\},\{2,3,4\}\}\right)\cup
\{[n]\setminus\{3,4\}\}.
\end{align}
Then $\mu_p(\FF)=\mu_p(\BD_3(n))-2p^3q^{n-3}+p^2q^{n-2}\to\mu_p(\BD_3(n))$
as $n\to\infty$, but $\FF$ is not contained in $\BD_3(n)$ (or any 
isomorphic copy of $\BD_3(n)$). 

Note that a non-trivial $r$-wise $t$-intersecting family is necessarily an 
$(r-1)$-wise $(t+1)$-intersecting family. (Otherwise there are $r-1$ 
subsets whose intersection is of size exactly $t$, and so all subsets
contain the $t$ vertices to be $r$-wise $t$-intersecting,
which contradicts the non-trivial condition.)
Thus Theorem~\ref{thm:EKL} also holds 
if we replace the assumption that $\FF$ is a 2-wise 2-intersecting family 
with the assumption that $\FF$ is a non-trivial 3-wise intersecting family.
We also note that the family $\FF$ defined by \eqref{2w2i not 3w1i}
is 2-wise 2-intersecting, but not 3-wise intersecting. 
This suggests a possibility of a stronger stability 
for non-trivial 3-wise intersecting families than
2-wise 2-intersecting families.
\begin{conj}
Let $\frac13<p\leq\frac12$. There is a constant $\epsilon_0=\epsilon_0(p)$ 
such that the following holds. If $\FF\subset 2^{[n]}$ is a non-trivial
$3$-wise intersecting family with $\mu_p(\FF)=\mu_p(\BD_3(n))-\epsilon$, 
where $\epsilon<\epsilon_0$, then there is a family $\GG\cong\BD_3(n)$ 
such that $\FF\subset\GG$.
\end{conj}
We verify the conjecture for the case $\frac25\leq p\leq\frac12$ 
provided that the family is shifted.
Here we say that a family $\FF\subset 2^{[n]}$ is shifted if
$F\in\FF$ and $\{i,j\}\cap\FF=\{j\}$ for some $1\leq i<j\leq n$, then
$(F\setminus\{j\})\cup\{i\}\in\FF$. The following is our main result in
this paper.
\begin{thm}\label{thm:main}
Let $\frac25\leq p\leq\frac12$, and let $\FF\subset 2^{[n]}$ be a shifted
non-trivial $3$-wise intersecting family. If $\FF\not\subset\BD_3(n)$
then $\mu_p(\FF)<\mu_p(\BD_3(n))-0.0018$.
\end{thm} 
For the proof of Theorem~\ref{thm:main} we divide the family into some
subfamilies. These subfamilies are not only 3-wise intersecting, but also
satisfy some additional intersection conditions. 
To capture the conditions we need some more definitions.
We say that $r$ families $\FF_1,\ldots,\FF_r\subset 2^{[n]}$ are 
$r$-cross $t$-intersecting if $|F_1\cap\cdots\cap F_r|\geq t$ for all
$F_1\in\FF_1,\ldots,F_r\in\FF_r$.
If moreover $\FF_i\neq\emptyset$ for all $1\leq i\leq r$, then 
the $r$ families are called non-empty $r$-cross $t$-intersecting.
As usual we say $r$-cross intersecting to mean $r$-cross $1$-intersecting.
The following result is used to prove Theorem~\ref{thm:main}.
\begin{thm}\label{thm3}
Let $\frac13\leq p\leq\frac12$. If $\FF_1,\FF_2,\FF_3\subset 2^{[n]}$ are 
non-empty $3$-cross intersecting families, then
\begin{align}\label{ineq<3p}
 \mu_p(\FF_1)+ \mu_p(\FF_2)+ \mu_p(\FF_3) \leq 3p. 
\end{align}
Suppose, moreover, that $\frac13<p\leq\frac12$, all $\FF_i$ are
shifted, and $\bigcap F=\emptyset$,
where the intersection is taken over all $F\in\FF_1\cup\FF_2\cup\FF_3$.
Then,
\begin{align}\label{ineq<3p-e}
 \mu_p(\FF_1)+ \mu_p(\FF_2)+ \mu_p(\FF_3) \leq 3p-\epsilon_p,
\end{align}
where $\epsilon_p=(2-3p)(3p-1)$.
\end{thm}
\noindent
The first inequality \eqref{ineq<3p} is an easy consequence of 
a recent result on $r$-cross $t$-intersecting families obtained by Gupta, 
Mogge, Piga, and Sch\"ulke \cite{GMPS}, while the second inequality 
\eqref{ineq<3p-e} is proved by solving linear programming (LP) problems.
We mention that equality holds in \eqref{ineq<3p} only if 
$|\bigcap_{F\in\FF_1\cup\FF_2\cup\FF_3} F|=1$ unless $p=\frac13$.
This fact will not be used for the proof of Theorem~\ref{thm:main}, but
it follows easily from \eqref{ineq<3p-e} and Lemma~\ref{shift-stable}.

Here we outline the proof of Theorem~\ref{thm:main}. This is done by
solving LP problems as follows. First we divide $\FF$ into subfamilies,
say, $\FF=\FF_1\cup\FF_2\cup\cdots\cup\FF_k$. Let $x_i=\mu_p(\FF_i)$.
These subfamilies satisfy some additional conditions, which give us
(not necessarily linear) constraints on the variables $x_i$. Under
the constraints we need to maximize $\mu_p(\FF)=\sum_{i=1}^kx_i$.
In principle this problem can be solved by the method of Lagrange
multipliers, but in practice it is not so easy even if concrete $p$ is 
fixed. To overcome the difficulty we first replace non-linear constraints
with weaker piecewise linear constraints. In this way the problem turns
into an LP problem with the parameter $p$, which can be solved efficiently
provided $p$ is given. Next, instead of solving this primal problem,
we seek a feasible solution to the dual LP problem with the parameter.
Finally the desired upper bound for $\mu_p(\FF)$ is obtained by the weak
duality theorem. 
As to related proof technique we refer
\cite{Wagner} for application of LP methods to some 
other extremal problems, and also \cite{ST,STT} for SDP methods.

In the next section we gather tools we use to prove our results.
Then in Section~3 we deduce Theorem~\ref{thm2} and 
Theorem~\ref{thm:M3(n,p)} from Theorem~\ref{thm:main}.
In Section~4 we prove Theorem~\ref{thm3}, whose proof is a prototype of
the proof of Theorem~\ref{thm:main}. In Section~5 we prove
our main result Theorem~\ref{thm:main}.
Finally in the last section we discuss possible extensions to non-trivial
$r$-wise intersecting families for $r\geq 4$, and a related $k$-uniform
problem. In particular we include counterexamples to a recent conjecture 
posed by O'Neill and Verstr\"aete \cite{OV} 
(c.f.~Balogh and Linz \cite{BL}).

\section{Preliminaries}\label{prelim}
\subsection{Shifting}
For $1\leq i<j\leq n$ we define the shifting operation 
$\sigma_{i,j}:2^{[n]}\to 2^{[n]}$ by 
\[
\sigma_{i,j}(\GG):=\{G_{i,j}:G\in\GG\}, 
\]
where
\[
G_{i,j}:=\begin{cases}
(G\setminus\{j\})\sqcup\{i\}&\text{if }(G\setminus\{j\})\sqcup\{i\}\not\in\GG,\\
G & \text{otherwise}.	   
\end{cases}
\]
By definition $\mu_p(\GG)=\mu_p(\sigma_{i,j}(\GG))$ follows.
We say that $\GG$ is shifted if $\GG$ is invariant under any shifting 
operations, in other words, if $G\in\GG$ then $G_{i,j}\in\GG$ for all 
$1\leq i<j\leq n$. If $\GG$ is not shifted then 
$\sum_{G\in\GG}\sum_{g\in G}g>\sum_{G'\in\sigma_{i,j}(\GG)}\sum_{g'\in G'}g'$
for some $i,j$, and so starting from $\GG$ we get a shifted $\GG'$ by applying
shifting operations repeatedly finitely many times.
It is not difficult to check that if $\GG$ is $r$-wise $t$-intersecting, 
then so is $\sigma_{i,j}(\GG)$. 
Therefore if $\GG$ is an $r$-wise $t$-intersecting family, 
then there is a shifted $r$-wise $t$-intersecting family $\GG'$ with 
$\mu_p(\GG')=\mu_p(\GG)$.
It is also true that if $\GG_1,\ldots,\GG_r$ are $r$-cross $t$-intersecting
families, then there are shifted $r$-cross $t$-intersecting families
$\GG'_1,\ldots,\GG_r'$ with $\mu_p(\GG_i)=\mu_p(\GG'_i)$ for all 
$1\leq i\leq r$.

For the proof of Theorem~\ref{thm:M3(n,p)} we use the fact that
if $\sigma_{i,j}(\GG)\cong\BD_3(n)$ then $\GG\cong\BD_3(n)$. 
More generally the following holds.
\begin{lemma}\label{shift-stable}
Let $n,a,b$ be positive integers with $a\geq 1$, $b\geq 0$, 
and $n\geq a+2b$, and let $\FF=\{F\subset[n]:|F\cap[a+2b]|\geq a+b\}$.
If $\GG\subset 2^{[n]}$ satisfies $\sigma_{i,j}(\GG)=\FF$ then 
$\GG\cong\FF$.
\end{lemma}
\noindent
The above result is well-known, see, e.g., Lemma~6 in \cite{LST} for a
proof and a history. 
We note that the condition $\sigma_{i,j}(\GG)=\FF$ can be replaced with
$\sigma_{i,j}(\GG)\cong\FF$. Indeed if $\sigma_{i,j}(\GG)=\FF'$ and
$\FF'\cong\FF$, then by Lemma~\ref{shift-stable} (and by renaming the
vertices) we have $\GG\cong\FF'$, and so $\GG\cong\FF$.
By choosing $a=r-1$ and $b=1$, we see that if 
$\sigma_{i,j}(\GG)\cong\BD_r(n)$ then $\GG\cong\BD_r(n)$.

For $G,H\subset[n]$ we say that $G$ shifts to $H$, denoted by $G\leadsto H$,
if $G=\emptyset$, or 
if $|G|\leq|H|$ and the $i$th smallest element of $G$ is greater than or equal
to that of $H$ for each $i\leq|G|$. 
Note that the relation $\leadsto$ is transitive, and this fact will be
used later (Claims~\ref{lemma:leadsto} and \ref{poset}).

We say that $\GG$ is inclusion maximal if $G\in\GG$ and $G\subset H$ imply $H\in\GG$.
Since we are interested in the maximum measure of non-trivial $3$-wise 
intersecting families, we always assume that families are inclusion maximal.
If $\GG$ is shifted and inclusion maximal, 
then $G\in\GG$ and $G\leadsto H$ imply $H\in\GG$.

\subsection{Duality in linear programming}
For later use we briefly record the weak duality theorem in linear
programming. See e.g., chapter~6 in \cite{GM} for more details.

A primal linear programming problem (P) is formalized as follows.
\begin{description}
\item[maximize] $c^\tp x$,
\item[subject to] $Ax\leq b$ and $x\geq 0$.
\end{description}
The corresponding dual programming problem (D) is as follows.
\begin{description}
\item[minimize] $b^\tp y$,
\item[subject to] $A^\tp y\geq c$ and $y\geq 0$.
\end{description}

\begin{thm}[Weak duality]\label{thm:weak duality}
 For each feasible solution $x$ of (P) and 
each feasible solution $y$ of (D) we have
$c^\tp x\leq b^\tp y$.
\end{thm}

\subsection{Tools for the proof of Theorem~\ref{thm3}}
Let $n,t,a$ be fixed positive integers with $t\leq a\leq n$.
Define two families $\AA$ and $\BB$ by
\begin{align*}
\AA&=\{F\subset[n]:|F\cap[a]|\geq t\}, \\
\BB&=\{F\subset[n]:[a]\subset F\}.
\end{align*}
Then $\mu_p(\AA)=1-\sum_{j=0}^{t-1}\binom ajp^jq^{a-j}$, and
$\mu_p(\BB)=p^a$.
Let $\FF_1=\AA$, $\FF_2=\cdots=\FF_r=\BB$.
Then $\FF_1,\ldots,\FF_r$ are $r$-cross $t$-intersecting families with
$\sum_{i=1}^r\mu_p(\FF_i)=\mu_p(\AA)+(r-1)\mu_p(\BB)$.
The next result is a special case of Theorem~1.4 in \cite{GMPS}, which
states that the above construction is the best choice to maximize the sum of
$p$-measures of non-empty $r$-cross $t$-intersecting families provided 
$p\leq\frac12$.

\begin{thm}[Gupta--Mogge--Piga--Sch\"ulke \cite{GMPS}]\label{thm4}
Let $r\geq 2$ and $0<p\leq\frac12$. If $\FF_1,\ldots,\FF_r\subset 2^{[n]}$ 
are non-empty $r$-cross $t$-intersecting families, then 
\[
\sum_{i=1}^r\mu_p(\FF_i)\leq
\max\bigg\{\bigg(1-\sum_{j=0}^{t-1}\binom ajp^jq^{a-j}\bigg)+(r-1)p^a
:t\leq a\leq n\bigg\}.
\]
\end{thm}
\noindent
We need the non-empty condition to exclude the case
$\FF_1=\emptyset$, $\FF_2=\cdots=\FF_r=2^{[n]}$.

\begin{lemma}\label{lemma:2c1i}
Let $0<p\leq\frac12$. Suppose that $\FF_1,\FF_2 \subset 2^{[n]}$ are 
$2$-cross intersecting families. 
\begin{enumerate}
\item[(i)] $\mu_p(\FF_1)+\mu_p(\FF_2)\leq 1$.
\item[(ii)] $\mu_p(\FF_1)\mu_p(\FF_2)\leq p^2$.
\end{enumerate}
\end{lemma}

\begin{proof}
(i) If one of the families is empty, then the inequality clearly holds.
So suppose that both families are non-empty. Then, by Theorem~\ref{thm4},
we have
\[
 \mu_p(\FF_1) +  \mu_p(\FF_2) \leq \max\{(1-q^a)+p^a:1\leq a\leq n\}.
\]
Thus it suffices to show that $1-q^a+p^a\leq 1$, or equivalently,
$p^a\leq (1-p)^a$ for all $a\geq 1$. Indeed this follows from the assumption
$p\leq \frac12$.

(ii) This is proved in \cite{T2010} as Theorem~2.
\end{proof}

\begin{lemma}\label{lemma:3w1i<=1}
Let $0<p\leq\frac12$, and $t\geq 2$. If $\FF_1,\FF_2, \FF_3 \subset 2^{[n]}$ 
are non-empty $3$-cross $t$-intersecting families, then 
$\mu_p(\FF_1)+\mu_p(\FF_2)+\mu_p(\FF_3)\leq 1$.
\end{lemma}

\begin{proof}
Note that if $t\geq 2$ then $3$-cross $t$-intersecting families are 
$3$-cross $2$-intersecting families. Thus, using Theorem~\ref{thm4}, 
it suffices to show that
\[
 f(p,a):=(1-q^a-apq^{a-1})+2p^a \leq 1
\]
for $a\geq t$.
This inequality follows from the fact that $f(p,a)$ is increasing in $p$,
and $f(\frac12,a)$ is non-decreasing in $a$ for $a=2,3,\ldots$, and
$\lim_{a\to\infty}f(\frac12,a)=1$.
\end{proof}

\subsection{Random walk}
Here we extend the random walk method to deal with $p$-measures of 
$r$-cross $t$-intersecting families possibly with different $p$-measures. 
The method was originally introduced by Frankl in \cite{Fshift}. 

Let $r\geq 2$ be a positive integer. For $1\leq i\leq r$ let
$p_i$ be a real number with $0<p_i<1-\frac1r$, and let $q_i=1-p_i$.
Let $\aa(p_i)\in(0,1)$ be a unique root of the equation
\begin{align}\label{eq:alpha}
X=p_i+q_iX^r,
\end{align}
and let $\bb=\bb(p_1,\ldots,p_r)\in(0,1)$ be a unique root of the equation 
\begin{align}\label{eq:beta}
 X=\prod_{i=1}^r(p_i+q_iX). 
\end{align} 

Consider two types of random walks, $A_i$ and $B$, in the two-dimensional
grid $\Z^2$. Both walks start at the origin, and at each step it moves
from $(x,y)$ to $(x,y+1)$ (one step up), or from $(x,y)$ to $(x+1,y)$
(one step to the right). For every step the type $A_i$ walk takes one step
up with probability $p_i$, and one step to the right with probability $q_i$.
On the other hand, at step $j$, the type $B$ walk takes one step up with
probability $p_i$, and one step to the right with probability $q_i$,
where $i=j\bmod r$. 
Let $L_j$ denote the line $y=(r-1)x+j$. 
\begin{claim}\label{claim:probability}
Let $r\geq 2$ and $t\geq 1$ be integers. Then we have
\begin{align*}
\P(\text{the type $A_i$ walk hits the line $L_t$})&=\aa(p_i)^t,\\
\P(\text{the type $B$ walk hits the line $L_{rt}$})&=\bb(p_1,\ldots,p_r)^t.
\end{align*}
\end{claim}
\begin{proof}
Let $x_i(t)$ denote the probability that the walk $A_i$ hits the line $L_t$.
After the first step of the walk, it is at $(0,1)$
with probability $p_i$, or at $(1,0)$ with probability $q_i$. From $(0,1)$
the probability for the walk hitting $L_t$ is $x_i(t-1)$, and from $(1,0)$
the probability is $x_i(t-1+r)$. Therefore we have
\begin{align}\label{recurrenceA}
 x_i(t)=p_i x_i(t-1)+q_i x_i(t-1+r). 
\end{align}
Let $a_j$ be the number of walks from $(0,0)$ to $P_j:=(j,(r-1)j+t)$
which touch $L_t$ only at $P_j$. (It is known that 
$a_j=\frac t{rj+t}\binom{rj+t}j$, but we do not need this fact.)
Then we have $x_i(t)=\sum_{j\geq 0}a_jp_i^{(r-1)j+t}q_i^j$.
If a walk touches the line $L_{t+1}$, then the walk needs to hit $L_t$ 
somewhere, say, at $P_j$ for the first time. 
Then the probability that the walk
hit $L_{t+1}$ starting from $P_j$ is equal to $x_i(1)$. Thus we have
\[
 x_i(t+1)=\sum_{j\geq 0}(a_jp_i^{(r-1)j+t}q_i^j)\,x_i(1)=x_i(t)x_i(1),
\]
and so we can write $x_i(t)=z^t$, where $z:=x_i(1)$.
Substituting this into \eqref{recurrenceA} and dividing both sides by 
$z^{t-1}$ we see that $z$ is a root of
the equation \eqref{eq:alpha}. Let $f(X):=p_i+q_iX^r-X$. Then we have
$f(0)=p_i>0$, $f(1)=0$, $f'(1)=q_ir-1>0$, and $f''(X)=q_ir(r-1)X^{r-2}>0$. 
Thus the equation $f(X)=0$, or equivalently, \eqref{eq:alpha} has 
precisely two roots in $[0,1]$, that is, $\alpha(p_i)$ and $1$. 
We claim that $z\neq 1$. Indeed we have
$\lim_{t\to\infty}x_i(t)=\lim_{t\to\infty}z^t=0$ because a step in 
the type $A_i$ walk reduces, on average, $y-(r-1)x$ by $(r-1)-rp_i>0$.
Consequently we have $z=\alpha(p_i)$, and so $x_i(t)=\alpha(p_i)^t$.

Next let $y(t)$ denote the probability that the walk $B$ hits the line 
$L_{rt}$. After the first $r$ steps, it is at $(x,r-x)$ for some
$0\leq x\leq r$ with probability
\[
 \sum_{J\in\binom{[r]}x}\prod_{i\in [r]\setminus J}p_i\prod_{j\in J}q_j.
\]
From $(x,r-x)$ the probability for the walk hitting $L_{rt}$ is $y(x+t-1)$.
This yields
\begin{align}\label{recurrenceB}
 y(t)=\sum_{x=0}^ry(x+t-1)
 \sum_{J\in\binom{[r]}x}\prod_{i\in [r]\setminus J}p_i\prod_{j\in J}q_j.
\end{align}
Let $b_s$ be the number of walks from $(0,0)$ to $Q_s:=(s,(r-1)s+rt)$
which touch $L_{rt}$ only at $Q_s$. Then we have 
\[
y(t)=\sum_{s\geq 0}b_s
\sum_{J\in\binom{[r(s+t)]}s}\prod_{i\in [r(s+t)]\setminus J}
p_i\prod_{j\in J}q_j.
\]
If a walk touches the line $L_{r(t+1)}$, then the walk needs to hit $L_{rt}$
somewhere, say, at $Q_s$ for the first time. 
Then the probability that the walk
hit $L_{r(t+1)}$ starting from $Q_s$ is equal to $y(1)$. Thus we have
\[
y(t+1)=
\sum_{s\geq 0}b_s
\sum_{J\in\binom{[r(s+t)]}s}\prod_{i\in [r(s+t)]\setminus J}
p_i\prod_{j\in J}q_j\,y(1)
=y(t)y(1),
\]
and so $y(t)=w^t$, where $w:=y(1)$.
Substituting this into \eqref{recurrenceB} and dividing both sides by 
$w^{t-1}$ we have
\begin{align*}
w=\sum_{x=0}^r w^x \sum_{J\in\binom{[r]}x}\prod_{i\in[r]\setminus J}p_i
\prod_{j\in J}q_j=\prod_{i=1}^r(p_i+q_iw).
\end{align*}
Thus $w$ is a root of the equation \eqref{eq:beta}.
Let $g(X):=\prod_{i=1}^r(p_i+q_iX)-X$. Then we have
$g(0)=\prod_i p_i>0$, $g(1)=0$, $g'(1)=\sum_i q_i>0$, and $g''(X)>0$. 
Thus the equation $g(X)=0$, or equivalently, \eqref{eq:beta} has 
precisely two roots in $[0,1]$, that is, $\beta$ and $1$. 
But we can exclude the possibility $w=1$ in the same way as in the previous
case. Thus we have $w=\beta$ and so $y(t)=\beta^t$.
\end{proof}

\begin{claim}\label{prop8.1}
Let $\FF_1,\ldots,\FF_r\subset 2^{[n]}$ be shifted $r$-cross $t$-intersecting
families. Then, for all $(F_1,\ldots,F_r)\in\FF_1\times\cdots\times\FF_r$, 
there exists $j=j(F_1,\ldots,F_r)\in[n]$ such that 
$\sum_{i=1}^r|F_i\cap[j]|\geq t+(r-1)j$.
\end{claim}

This is Proposition 8.1 in \cite{Fshift}. We include a simple proof
for convenience.
\begin{proof}
Suppose the contrary. Then there exist an $r$-tuple of a counterexample
$(F_1,\ldots,F_r)\in\FF_1\times\cdots\times\FF_r$, which we choose
$|F_1\cap\cdots\cap F_r|$ minimal. Let $j$ be the $t$-th element of
$F_1\cap\cdots\cap F_r$. Then we have
\[
\sum_{i=1}^r|F_i\cap[j]|<t+(r-1)j=|F_1\cap\cdots\cap F_r\cap[j]|+(r-1)|[j]|.
\]
Thus there exist some $i\in[j-1]$ such that $i$ is not contained in (at least)
two of the $r$ subsets, say, $i\not\in F_1\cup F_2$. By the shiftedness
we have $F_1':=(F\setminus\{j\})\cup\{i\}\in\FF_1$. Then 
$|F_1\cap[j]|=|F_1'\cap[j]|$ and so $(F_1',F_2,\ldots,F_r)$ is also a
counterexample. But this contradicts the minimality because
$|F_1'\cap F_2\cap \cdots\cap F_r|<|F_1\cap F_2\cap\cdots\cap F_r|$.
\end{proof}

Let $\FF_1,\ldots,\FF_r\subset 2^{[n]}$ be families of subsets. For each
$(F_1,\ldots,F_r)\in\FF_1\times\cdots\times\FF_r$ we define a vector $w$ by
\[
w=w(F_1,\ldots,F_r):=(w_1^{(1)},w_2^{(1)},\ldots,w_r^{(1)},\ldots,
w_1^{(n)},w_2^{(n)},\ldots,w_r^{(n)})\in\{0,1\}^{rn},
\]
where 
\[
 w_i^{(j)}=\begin{cases}
  1&\text{ if }j\in F_i, \\
  0&\text{ if }j\not\in F_i.
  \end{cases}
\]
We can view $w$ as an $rn$-step walk whose $k$-th step is up (resp.\ right)
if the $k$-th entry of $w$ is $1$ (resp.\ $0$) for $1\leq k\leq rn$.

\begin{claim}\label{claim:walk}
Let $\FF_1,\ldots,\FF_r\subset 2^{[n]}$ be shifted $r$-cross $t$-intersecting
families. 
Then, for all $(F_1,\ldots,F_r)\in\FF_1\times\cdots\times\FF_r$, 
the walk $w(F_1,\ldots,F_r)$ hits the line $L_{rt}$.
\end{claim}

\begin{proof}
Let $j=j(F_1,\ldots,F_r)$ be from Claim~\ref{prop8.1}, and let 
$w=w(F_1,\ldots,F_r)$ be the corresponding walk. 
In the first $rj$ steps of $w$ there are at least $t+(r-1)j$ up steps,
and so at most $rj-(t+(r-1)j)=j-t$ right steps. 
This means that the walk $w$ hits the line $L_{rt}$ within the first
$rj$ steps.
\end{proof}

\begin{thm}
Let $p_1,\ldots,p_r$ be positive real numbers less than $1-\frac1r$, and
let $\FF_1,\ldots,\FF_r\subset 2^{[n]}$ be $r$-cross $t$-intersecting
families. 
Then we have $\prod_{i=1}^r\mu_{p_i}(\FF_i)\leq \beta^t$,
where $\beta$ is the root of the equation \eqref{eq:beta}.
\end{thm}

\begin{proof}
Since the shifting operation preserves $r$-cross $t$-intersecting 
property and $p$-measures, we may assume that all $\FF_i$ are shifted.
We have
\[
\prod_{i=1}^r\mu_{p_i}(\FF_i) 
 = \prod_{i=1}^r\sum_{F_i\in\FF_i}p_i^{|F_i|}q_i^{n-|F_i|}\\
 = \sum_{(F_1,\ldots,F_r)\in\FF_1\times\cdots\times\FF_r}
\prod_{i=1}^rp_i^{|F_i|}q_i^{n-|F_i|}.
\]
Using Claim~\ref{claim:walk} the RHS is
\[
\leq\P(\text{type $B$ walk hits $L_{rt}$ in the first $rn$ steps})
\leq \P(\text{type $B$ walk hits $L_{rt}$}) = \beta^t,
\]
where the last equality follows from Claim~\ref{claim:probability}.
\end{proof}

By comparing \eqref{eq:alpha} and \eqref{eq:beta} it follows that
if $p_1=\cdots=p_r=:p$ then $\beta(p,\ldots,p)=\alpha(p)^r$.
If $r=3$ then it is not so difficult to verify that 
$\beta(p_1,p_2,p_3)\leq\alpha(p_1)\alpha(p_2)\alpha(p_3)$, 
see \cite{KKT} for more details, and we have the following.  

\begin{lemma}\label{la2}
Let $0<p_1,p_2,p_3<\frac23$ and $t$ be a positive integer.  
If $\FF_1,\FF_2,\FF_3\subset 2^{[n]}$ are $3$-cross $t$-intersecting, then
\[
 \mu_{p_1}(\FF_1) \mu_{p_2}(\FF_2) \mu_{p_3}(\FF_3) \leq 
(\alpha(p_1)\alpha(p_2)\alpha(p_3))^t,
\]
where
\begin{align}\label{def:alpha}
 \alpha(p):=\frac12\left(\sqrt{\frac{1+3p}{1-p}}-1\right). 
\end{align}
\end{lemma}

\subsection{Tools for the proof of Theorem~\ref{thm:main}}
Let $0<p_1<p_2<1$ be fixed. 
Let $\R^{[p_1,p_2]}$ denote the set of real-valued functions defined on
the interval $[p_1,p_2]:=\{x\in\R:p_1\leq x\leq p_2\}$.
We will bound a convex function $g\in\R^{[p_1,p_2]}$ by a linear function
connecting $(p_1,g(p_1))$ and $(p_2,g(p_2))$. To this end, define an 
operator $L_{p_1,p_2}:\R^{[p_1,p_2]}\to\R^{[p_1,p_2]}$ by
\[
 (L_{p_1,p_2}(g))(p):=\frac{g(p_2)-g(p_1)}{p_2-p_1}(p-p_1)+g(p_1).
\]
By definition we have the following.
\begin{claim}\label{claim:piecewiselinear}
Let $g\in\R^{[p_1,p_2]}$ be a convex function. Then 
$g(p)\leq (L_{p_1,p_2}(g))(p)$ for $p\in[p_1,p_2]$.
\end{claim}

The function $\alpha=\alpha(p)$ defined by \eqref{def:alpha} is convex 
because $\frac{\partial^2\alpha(p)}{\partial p^2}
=\frac{6p}{(1+3p)^2q^2}\big(\frac{1+3p}q \big)^{1/2}>0$. 
Thus by Claim~\ref{claim:piecewiselinear} we have the following.
\begin{claim}\label{tilde-alpha}
For $\frac25\leq p\leq\frac12$ it follows that 
$\alpha(p)\leq\tilde\alpha(p)$, where
\begin{align*}
 \tilde\alpha(p)&:=(L_{\frac25,\frac12}(\alpha))(p)=
(-3 - 12\sqrt{5} + 5\sqrt{33})/6 + (30\sqrt{5} - 10\sqrt{33})p/6\\
&\approx -0.185 + 1.60607 p.
\end{align*}
\end{claim}

Let $\AK(n,t,p)$ denote the maximum $p$-measure $\mu_p(\GG)$ 
of 2-wise $t$-intersecting families $\GG\subset 2^{[n]}$. 

\begin{thm}[Ahlswede and Khachatrian \cite{AK}]\label{thm:AK}
Let
\[
 \frac i{t+2i-1}\leq p\leq \frac{i+1}{t+2i+1}.
\]
Then $\AK(n,t,p)=\mu_p(\AA(n,t,i))$, where
\[
 \AA(n,t,i)=\{A\subset[n]:|A\cap[t+2i]|\geq t+i\}.
\]
Moreover, if $\frac i{t+2i-1}< p< \frac{i+1}{t+2i+1}$ 
(resp.\ $p=\frac i{t+2i-1}$) then 
$\mu_p(\GG)=\AK(n,t,p)$ if and only if $\GG\cong\AA(n,t,i)$
(resp.\ $\GG\cong \AA(n,t,i-1)$ or $\GG\cong\AA(n,t,i)$).
\end{thm}

Let 
\[
 f_t(p):=\limsup_{n\to\infty} \AK(n,t,p).
\]
By Katona's $t$-intersection theorem we have $f_t(\frac12)=\frac12$.
For $p<\frac12$, by Theorem~\ref{thm:AK}, we have
$\AK(n,t,p)=\max\{\mu_p(\AA(n,t,i)):i\leq\frac{n-t}2\}$, and
$\AK(n,t,p)$ is non-decreasing in $n$. In this case we have
\[
 f_t(p)=\lim_{n\to\infty}\AK(n,t,p)=\sum_{j=t+i}^{t+2i}
 \binom{t+2i}jp^jq^{t+2i-j},
\]
where $i=\left\lfloor\frac{(t-1)p}{1-2p}\right\rfloor$.
The function $f_t(p)$ is left-continuous at $p=\frac12$.

\begin{claim}\label{ft is convex}
Let $t\geq 2$ be fixed. Then $f_t(p)$ is a convex function in $p$.
\end{claim}
\begin{proof}
First suppose that $\frac i{t+2i-1}<p<\frac{i+1}{t+2i+1}$.
Then $f_t(p)=\sum_{t+i}^{t+2i}\binom{t+2i}jp^jq^{t+2i-j}=:g(p)$, and we have
\[
\frac{\partial^2}{\partial p^2}\,f_t(p)=
\frac{(2i+t)!}{(i+t-1)!\,i!}\,p^{t-2+i}q^{i-1} (i+t-1-(2i+t-1)p) > 0.
\]
Next let $p_0=\frac i{t+2i-1}$. If $p$ is slightly larger than $p_0$
then we have the same $f_t(p)=g(p)$ as above, and if $p$ is slightly smaller
than $p_0$ then $f_t(p)=\sum_{t+i-1}^{t+2i-2}p^jq^{t+2i-2}=:h(p)$.
Since $h(p)<g(p)$ for $p>p_0$ and $h(p)>g(p)$ for $p<p_0$, we see that
the left derivative of $f_t(p)$ at $p=p_0$ is smaller than that of
the right derivative.
\end{proof}

By Claim~\ref{claim:piecewiselinear} and Claim~\ref{ft is convex} we have
the following.
\begin{claim}\label{tilde-at}
For $\frac25\leq p\leq\frac12$ it follows that 
$\AK(n,t,p)\leq\tilde a_t(p)$, where
$\tilde a_t=L_{\frac25,\frac12}(f_t)$.
\end{claim}
For convenience we record the $\tilde a_t$ which will be used to 
prove Theorem~\ref{thm:main}.
\begin{align*}
 \tilde a_2(p) & = \tfrac12 + (401(p-\tfrac12))/125\approx -1.104 + 3.208 p,\\
 \tilde a_3(p) & = \tfrac12 + (1565029(p-\tfrac12))/390625\approx -1.50324 + 4.00647 p,\\
 \tilde a_4(p) & = \tfrac12 + (5391614441(p-\tfrac12))/1220703125
\approx -1.70841 + 4.41681 p
,\\
 \tilde a_5(p) & = \tfrac12 + (17729648464189(p-\tfrac12))/3814697265625
\approx -1.82386 + 4.64772 p.
\end{align*}

\section{Proof of Theorem~\ref{thm2} and Theorem~\ref{thm:M3(n,p)}}
In this section we deduce Theorem~\ref{thm2} from Theorem~\ref{thm:M3(n,p)},
and then deduce Theorem~\ref{thm:M3(n,p)} from Theorem~\ref{thm:main}
whose proof is given in the next section.

\begin{proof}[Proof of Theorem~\ref{thm2}]
Let $\FF\subset 2^{[n]}$ be a non-trivial 3-wise intersecting family
with $\mu_p(\FF)=M_3(n,p)$. We may assume that $\FF$ is shifted and
inclusion maximal.
Since $\FF$ is {\em non-trivially} 3-wise intersecting, it is also $2$-wise
 $2$-intersecting, and so $M_3(n,p)\leq\AK(n,2,p)$. 

\begin{claim}
If $p<\frac13$ then $M_3(p)=p^2$.
\end{claim}
\begin{proof}
Let $p<\frac13$ be fixed. 
Then we have $\mu_p(\FF)=M_3(n,p)\leq\AK(n,2,p)=p^2$.
Moreover $\GG\cong\AA(n,2,0)$ is the only 2-wise 2-intersecting family 
with $\mu_p(\GG)=p^2$. Since $\AA(n,2,0)$ is not non-trivial 
3-wise intersecting, we get $M_3(n,p)<p^2$. 

On the other hand we can construct a non-trivial 3-wise 
intersecting family $\FF_1$ by
\[
 \FF_1=\{F\in[n]:[2]\subset F,\,F\cap[3,n]\neq\emptyset\}\sqcup
\{[n]\setminus\{1\}\}\sqcup\{[n]\setminus\{2\}\}.
\]
Then it follows that
\[
 \mu_p(\FF_1)=p^2(1-q^{n-2})+2p^{n-1}q \to p^2 \text{ as } n\to\infty.
\]
Thus we have $M_3(p)=p^2$ for $p<\frac13$.
\end{proof}

\begin{claim}
If $\frac13\leq p\leq\frac12$ then $M_3(p)=4p^3q+p^4$.
\end{claim}
\begin{proof}
This is an immediate consequence of Theorem~\ref{thm:M3(n,p)}.
\end{proof}

\begin{claim}
If $\frac12<p\leq\frac23$ then $M_3(p)=p$.
\end{claim}

\begin{proof}
Let $\frac12<p\leq\frac23$ be fixed. It is known from 
\cite{FGL,FT2003,T2021} that $3$-wise intersecting families $\GG$ have 
$p$-measure at most $p$ for $p\leq\frac23$, and moreover if 
$\mu_p(\GG)=p$ then $|\bigcap\GG|=1$ for $p<\frac23$. Thus we have 
$M(n,p)<p$ for $\frac12<p<\frac23$ and $M(n,\frac23)\leq\frac23$.

On the other hand, let us define a non-trivial 3-wise intersecting family
$\FF_2$ by
\[
 \FF_2=\{F\in[n]:1\in F,\,|F\cap[2,n]|\geq n/2\}\sqcup\{[2,n]\}.
\]
Then it follows that, for fixed $p$,
\[
 \mu_p(\FF_2)=p\sum_{k\geq n/2}\binom{n-1}kp^kq^{n-1-k}+qp^{n-1}\to p \text{ as } n\to\infty.
\]
Thus we have $M(p)=p$ for $\frac12 <p\leq\frac23$.
\end{proof}

\begin{claim}
If $\frac23<p<1$ then $M_3(p)=1$.
\end{claim}

\begin{proof}
Let $\frac23<p<1$ be fixed. Clearly we have $M(n,p)\leq 1$ and $M(p)\leq 1$.
Let us define a non-trivial 3-wise intersecting family $\FF_3$ by
\[
 \FF_3=\{F\subset[n]:|F|>\tfrac23n\}.
\]
Then $\mu_p(\FF_3)=\sum_{i>\frac23n}\binom nip^iq^{n-i}\to 1$ as
$n\to\infty$. Thus we have $M(p)=1$ for $p>\frac23$.
\end{proof}
This completes the proof of Theorem~\ref{thm2} assuming 
Theorem~\ref{thm:M3(n,p)}.
\end{proof}

\begin{proof}[Proof of Theorem~\ref{thm:M3(n,p)}]
Let $\FF\subset 2^{[n]}$ be a non-trivial 3-wise intersecting family.

First suppose that $\frac13\leq p<\frac25$. 
Note that $\FF$ is 2-wise 2-intersecting,
and $\AA(n,2,1)=\BD_3(n)$. Thus it follows from Theorem~\ref{thm:AK} that 
$\mu_p(\FF)\leq\mu_p(\BD_3(n))$. Moreover equality holds if and only if
$\FF\cong\BD_3(n)$ for $p>\frac13$. If $p=\frac13$ then 
$\mu_p(\FF)=\mu_p(\BD_3(n))$ if and only if $\FF\cong\BD_3(n)$ or
$\AA(n,2,0)$, but the latter is not non-trivial 3-wise intersecting, and so 
$\FF\cong\BD_3(n)$ must hold.

Next suppose that $\frac25\leq p\leq \frac12$. If $\FF$ is shifted then
by Theorem~\ref{thm:main} we have $\mu_p(\FF)\leq\mu_p(\BD_3(n))$ with
equality holding if and only if $\FF=\BD_3(n)$. The same inequality 
holds without assuming that $\FF$ is shifted (see the first paragraph
in Section~\ref{prelim}). In this case, by Lemma~\ref{shift-stable},
we have $\mu_p(\FF)=\mu_p(\BD_3(n))$ if and only if $\FF\cong\BD_3(n)$.
\end{proof}

\section{Proof of Theorem~\ref{thm3}}

\subsection{Proof of (\ref{ineq<3p}) of Theorem~\ref{thm3}}
Let $\frac13\leq p\leq \frac12$, and let $\FF_1,\FF_2,\FF_3$ be non-empty
3-cross intersecting families. By Theorem~\ref{thm4} with $r=3$ we have
$\sum_{i=1}^3\mu_p(\FF_i)\leq\max\{(1-q^a)+2p^a:a\in[n]\}$.
So we need to show that $f(p,a)\geq 0$, where 
\[
f(p,a):=3p-(1-q^a)-2p^a, 
\]
for all $\frac13\leq p\leq\frac12$ and all $1\leq a\leq n$.

If $a=1$ then $f(p,1)=3p-(1-q)-2p=0$, and we are done. So we may assume
that $2\leq a\leq n$, and we show that $f(p,a)>0$.

If $p=\frac13$ then $f(\frac13,a)=1-1+(\frac23)^a-2(\frac13)^a=
(\frac13)^a(2^a-2)>0$.
We claim that $f(p,a)$ is increasing in $p$, 
which yields $f(p,a)\geq f(\frac13,a)>0$ (for $a\geq 2$). We have
\begin{align}\label{eq:f'(a,p)}
 \frac{\partial f}{\partial p}(p,a)=3-aq^{a-1}-2ap^{a-1}. 
\end{align}
Fix $p$ and let $g(a)$ denote the RHS of 
\eqref{eq:f'(a,p)}. We have $g(2)=1-2p>0$ for $\frac13\leq p<\frac12$.
Next we show that $g(a)$ is increasing in $a$. For this we have
\[
 g(a+1)-g(a)=(ap-q)q^{a-1}+2(aq-p)p^{a-1},
\]
and we need to show that the RHS is positive.
Since $aq-p\geq ap-q$ it suffices to show that $ap-q\geq 0$,
or equivalently, $a\geq\frac {1-p}p$. Indeed $a\geq 2\geq\frac{1-p}p$ 
because $p\geq\frac13$.
Thus $g(a)$ is increasing in $a$, and $g(a)\geq g(2)>0$ as needed.
\qed

\subsection{Proof of (\ref{ineq<3p-e}) of Theorem~\ref{thm3}}
Recall that, for $i<j$, 
we write $[i,j]:=\{i,i+1,\ldots,j\}=[j]\setminus[i-1]$.

We divide $\FF_i=\{\{1\}\sqcup A:A\in\AA_i\}\cup\BB_i$, where
\begin{align*}
 \AA_i&:=\{F\setminus\{1\}:1\in F\in\FF_i\}\subset 2^{[2,n]},\\
 \BB_i&:=\{F:1\not\in F\in\FF_i\}\subset 2^{[2,n]}.
\end{align*}
Since $\FF_i\neq\emptyset$ is shifted, we have $\AA_i\neq\emptyset$.
Let $a_i=\mu_p(\AA_i:[2,n])>0$ and $b_i=\mu_p(\BB_i:[2,n])\geq 0$.
Then $\sum_{i=1}^3\mu_p(\FF_i)=\sum_{i=1}^3(pa_i+qb_i)$.
Without loss of generality we may assume that $b_1\geq b_2\geq b_3$.
If $b_1=0$ then $B_i=\emptyset$ for all $i$. 
In this case $1\in\bigcap F$, where the intersection is taken over all
$F\in\FF_1\cup\FF_2\cup\FF_3$, a contradiction.
So we may assume that $b_1\neq 0$, that is, $\BB_1\neq\emptyset$.
\begin{claim}
Let $\{i,j,k\}=[3]$.
\begin{enumerate}
\item If $\{\AA_i,\AA_j,\BB_k\}$ are all non-empty, then they are $3$-cross
intersecting, and $a_i+a_j+b_k\leq 3p$.
\item If $\{\AA_i,\BB_j,\BB_k\}$ are all non-empty, then they are $3$-cross
$2$-intersecting, and $a_i+b_j+b_k\leq 1$.
\end{enumerate} 
\end{claim}
\begin{proof}
The item (1) follows from the assumption that 
$\FF_i,\FF_j,\FF_k$ are 3-cross intersecting, and \eqref{ineq<3p} of
Theorem~\ref{thm3}.

To show (2), suppose, to the contrary, that there exist three subsets
$A_i\in\AA_i$, $B_j\in\BB_j$, $B_k\in\BB_k$, and $x\in[2,n]$ such that
$\{x\}\supset A_i\cap B_j\cap B_k$. 
By definition we have $F_i:=\{1\}\cup A_i\in\FF_i$ and $F_k:=B_k\in\FF_k$.
By the shiftedness we have $F_j:=(B_j\setminus\{x\})\cup\{1\}\in\FF_j$. 
Then $F_i\cap F_j\cap F_k=\emptyset$, a contradiction.
Thus $\{\AA_i,\BB_j,\BB_k\}$ are 3-wise 2-intersecting, and the inequality
follows from Lemma~\ref{lemma:3w1i<=1}.
\end{proof}

Now we will show that $\sum_{i=1}^3\mu_p(\FF_i)\leq 3p-\epsilon_p$,
where $\epsilon_p=(2-3p)(3p-1)$.

\subsubsection{Case $\BB_2=\BB_3=\emptyset$}
Since $\{\BB_1,\AA_2,\AA_3\}$ are $3$-cross intersecting, any two of 
them are 2-cross intersecting. Then, by \eqref{ineq<3p} of
Theorem~\ref{thm3} and Lemma~\ref{lemma:2c1i}, we have the following.

\begin{claim}\label{claim:b2=b3=0}
\begin{enumerate}
\item $b_1+a_2+a_3\leq 3p$,
\item $b_1+a_2\leq 1$,
\item $b_1+a_3\leq 1$.
\end{enumerate}
\end{claim}

We solve the following linear programming problem:
\begin{description}
\item[maximize] $\sum_{i=1}^3\mu_p(\FF_i)=p(a_1+a_2+a_3)+qb_1$,
\item[subject to] (1)--(3) in Claim~\ref{claim:b2=b3=0}, 
and $0\leq a_i\leq 1$ for all $i$, $0\leq b_1\leq 1$.
\end{description}
The corresponding dual problem is
\begin{description}
\item[minimize] $3py_1+\sum_{i=2}^7 y_i$,
\item[subject to]
$y_4\geq p$,
$y_1+y_2+y_5\geq p$,
$y_1+y_3+y_6\geq p$,
$y_1+y_2+y_3+y_7\geq q$,
and $y_i\geq 0$ for all $i$. 
\end{description}
\begin{table}[h]
\caption{Case $\BB_2=\BB_3=\emptyset$}
\begin{tabular}{c|cccc|c}
&$a_1$&$a_2$&$a_3$&$b_1$&\\
\hline
$y_1$&&1&1&1&$3p$\\
$y_2$&&1&&1&$1$\\
$y_3$&&&1&1&$1$\\
$y_4$&1&&&&$1$\\
$y_5$&&1&&&$1$\\
$y_6$&&&1&&$1$\\
$y_7$&&&&1&$1$\\
\hline
&$p$&$p$&$p$&$q$&
\end{tabular}
\end{table}
A feasible solution is given by
$y_1=3p-1$,
$y_2=y_3=1-2p$,
$y_4=p$,
$y_5=y_6=y_7=0$,
and the corresponding value of the objective function is
\[
3p(3p-1)+2(1-2p)+p=2-6p+9p^2=3p-\epsilon_p. 
\]
Then it follows from Theorem~\ref{thm:weak duality} (weak duality)
that the same bound applies to the primal problem, and so 
$\sum_{i=1}^3\mu_p(\FF_i)\leq 3p-\epsilon_p$.

\subsubsection{Case $\BB_2\neq\emptyset$, $\BB_3=\emptyset$}
In this case $\{\AA_1,\BB_2,\AA_3\}$ and $\{\BB_1,\AA_2,\AA_3\}$
are both 3-cross intersecting, 
and $\{\BB_1,\BB_2,\AA_3\}$ are 3-cross 2-intersecting.
Thus we have the following.
\begin{claim}\label{claim:b2!=0,b3=0}
\begin{enumerate}
\item $b_1+a_2+a_3\leq 3p$,
\item $a_1+b_2+a_3\leq 3p$,
\item $b_1+b_2+a_3\leq 1$,
\item $b_1+a_2\leq 1$,
\item $a_1+b_2\leq 1$.
\end{enumerate}
\end{claim}

We solve the following linear programming problem:
\begin{description}
\item[maximize] $\sum_{i=1}^3\mu_p(\FF_i)=p(a_1+a_2+a_3)+q(b_1+b_2)$,
\item[subject to] (1)--(5) in Claim~\ref{claim:b2!=0,b3=0}, 
and $0\leq a_i\leq 1$ for all $i$, $0\leq b_j\leq 1$ for all $j$.
\end{description}
The corresponding dual problem is
\begin{description}
\item[minimize] $3p(y_1+y_2)+\sum_{i=3}^{10} y_i$,
\item[subject to]
$y_2+y_5+y_6\geq p$,
$y_1+y_4+y_7\geq p$,
$y_1+y_2+y_3+y_8\geq p$,
$y_1+y_3+y_4+y_9\geq q$,
$y_2+y_3+y_5+y_{10}\geq q$,
and $y_i\geq 0$ for all $i$. 
\end{description}
\begin{table}[h]
\caption{Case $\BB_2\neq\emptyset, \BB_3=\emptyset$}
\begin{tabular}{c|ccccc|c}
&$a_1$&$a_2$&$a_3$&$b_1$&$b_2$\\
\hline
$y_1$&&1&1&1&&$3p$\\
$y_2$&1&&1&&1&$3p$\\
$y_3$&&&1&1&$1$&1\\
$y_4$&&1&&$1$&&1\\
$y_5$&1&&&&$1$&1\\
$y_6$&1&&&&&$1$\\
$y_7$&&1&&&&$1$\\
$y_8$&&&1&&&$1$\\
$y_9$&&&&1&&$1$\\
$y_{10}$&&&&&1&$1$\\
\hline
&$p$&$p$&$p$&$q$&$q$&
\end{tabular}
\end{table}
A feasible solution is given by
$y_1=y_6=y_7=y_8=y_9=y_{10}=0$,
$y_2=3p-1$,
$y_3=y_5=1-2p$,
$y_4=p$,
and the corresponding value of the objective function is
\[
3p(3p-1)+2(1-2p)+p=3p-\epsilon_p.
\]
Thus, by the weak duality, we have 
$\sum_{i=1}^3\mu_p(\FF_i)\leq 3p-\epsilon_p$.

\subsubsection{Case $\BB_2\neq\emptyset$, $\BB_3\neq\emptyset$}
Let $\{i,j,k\}=[3]$. 
Then families $\{\AA_i,\AA_j,\BB_k\}$ are 3-cross intersecting, and 
families $\{\AA_i,\BB_j,\BB_k\}$ are 3-cross 2-intersecting.
Thus we have the following.
\begin{claim}\label{claim:b3!=0}
\begin{enumerate}
\item $b_1+a_2+a_3\leq 3p$,
\item $a_1+b_2+a_3\leq 3p$,
\item $a_1+a_2+b_3\leq 3p$,
\item $a_1+b_2+b_3\leq 1$,
\item $b_1+a_2+b_3\leq 1$,
\item $b_1+b_2+a_3\leq 1$.
\end{enumerate}
\end{claim}

We solve the following linear programming problem:
\begin{description}
\item[maximize] $\sum_{i=1}^3\mu_p(\FF_i)=p(a_1+a_2+a_3)+q(b_1+b_2+b_3)$,
\item[subject to] (1)--(6) in Claim~\ref{claim:b3!=0}, 
and $0\leq a_i\leq 1$ for all $i$, $0\leq b_j\leq 1$ for all $j$.
\end{description}
The corresponding dual problem is
\begin{description}
\item[minimize] $3p(y_1+y_2+y_3)+\sum_{i=4}^{12} y_i$,
\item[subject to]
$y_2+y_3+y_4+y_7\geq p$,
$y_1+y_3+y_5+y_8\geq p$,
$y_1+y_2+y_6+y_9\geq p$,
$y_1+y_5+y_6+y_{10}\geq q$,
$y_2+y_4+y_6+y_{11}\geq q$,
$y_3+y_4+y_5+y_{12}\geq q$,
and $y_i\geq 0$ for all $i$. 
\end{description}
\begin{table}[h]
\caption{Case $\BB_2\neq\emptyset,\BB_3\neq\emptyset$}
\begin{tabular}{c|cccccc|c}
&$a_1$&$a_2$&$a_3$&$b_1$&$b_2$&$b_3$\\
\hline
$y_1$&&1&1&1&&&$3p$\\
$y_2$&1&&1&&1&&$3p$\\
$y_3$&1&1&&&&$1$&$3p$\\
$y_4$&1&&&&1&1&1\\
$y_5$&&1&&1&&1&1\\
$y_6$&&&1&1&1&&1\\
$y_7$&1&&&&&&$1$\\
$y_8$&&1&&&&&$1$\\
$y_9$&&&1&&&&$1$\\
$y_{10}$&&&&1&&&$1$\\
$y_{11}$&&&&&1&&$1$\\
$y_{12}$&&&&&&1&$1$\\
\hline
&$p$&$p$&$p$&$q$&$q$&$q$&
\end{tabular}
\end{table}
A feasible solution is given by
$y_1=3p-1$,
$y_2=y_3=y_7=y_8=y_9=y_{10}=y_{11}=y_{12}=0$,
$y_4=p$,
$y_5=y_6=1-2p$,
and the corresponding value of the objective function is
\[
3p(3p-1)+p+2(1-2p)=3p-\epsilon_p.
\]
Thus we have $\sum_{i=1}^3\mu_p(\FF_i)\leq 3p-\epsilon_p$.

This complete the proof of \eqref{ineq<3p-e} of Theorem~\ref{thm3}. \qed

\section{Proof of Theorem~\ref{thm:main}}
Let $\frac25\leq p\leq\frac12$, and let $\FF\subset 2^{[n]}$ be a 
non-trivial 3-wise intersecting family. Suppose that $\FF$ is shifted, 
inclusion maximal, and $\FF\not\subset\BD_3(n)$. 
We may also assume that $\FF$ is size maximal (with respect to
3-wise intersection condition), that is, for every
$G\not\in\FF$, the larger family $\FF\cup\{G\}$ is no longer
3-wise intersecting. Our goal is to show that 
\[
\mu_p(\FF)<\mu_p(\BD_3(n))-0.0018. 
\]
For $I\subset[3]$ define $\FF_I\subset 2^{[n]}$ and 
$\GG_I\subset 2^{[4,n]}$ by
\begin{align*}
 \FF_I&=\{F\in\FF:F\cap[3]=I\},\\
 \GG_I&=\{F\setminus[3]:F\in\FF_I\}.
\end{align*}
Let $x_I=\mu_p(\GG_I:[4,n])$. Then we have
\[
 \mu_p(\FF_I)=p^{|I|}q^{3-|I|}x_I,
\]
and 
\begin{align}\label{mu(FF)}
 \mu_p(\FF) =\sum_{I\subset[3]} p^{|I|}q^{3-|I|}x_I. 
\end{align} 
For simplicity we often write $\GG_I$ or $x_I$ without braces and 
commas, e.g., we write $\GG_{12}$ to mean $\GG_{\{1,2\}}$.
Let $\bar I$ denote $[3]\setminus I$.

\begin{claim}\label{lemma:leadsto}
If $I,J\subset[3]$ satisfy $I\leadsto J$ then $\GG_I\subset\GG_J$ and 
$x_I\leq x_J$.
\end{claim}
\begin{proof}
Suppose that $G\in\GG_I$. Then $I\cup G\in\FF_I$.
Since $\FF$ is shifted, inclusion maximal, and $I\leadsto J$, we have that
$J\cup G\in\FF_J$, and so $G\in\GG_J$.
Thus $\GG_I\subset\GG_J$, and so $x_I\leq x_J$.
\end{proof}

Applying Claim~\ref{lemma:leadsto} to the diagram in Figure~\ref{fig1},
we get Claim~\ref{poset}.
\begin{figure}[h]
\includegraphics[scale=1.2]{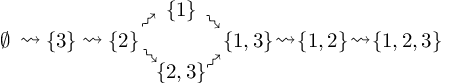}  
\caption{Poset induced by shifting and inclusion}\label{fig1}
\end{figure}

\begin{claim}\label{poset}
We have
$x_\emptyset\leq x_3\leq x_2\leq x_1\leq x_{13}\leq x_{12}\leq x_{123}$,
and $x_2\leq x_{23}\leq x_{13}$.
\end{claim}

Let $I_1,I_2,I_3\subset[3]$. Define a $3\times 3$ matrix 
$M=M(I_1,I_2,I_3)=(m_{i,j})$ by
\[
 m_{i,j}=\begin{cases}
	  1 & \text{if }j\in I_i,\\
	  0 & \text{if }j\not\in I_i.
	 \end{cases}
\]
Then $I_1\cap I_2\cap I_3=\emptyset$ if and only if every column of $M$
contains (at least one) $0$. In this case we say that $M$ is acceptable,
and let $\tau:=7-s$, where $s$ is the total sum of $m_{i,j}$.

\begin{claim}\label{M}
Let $M(I_1,I_2,I_3)$ be acceptable. 
If $\{\GG_{I_1},\GG_{I_2},\GG_{I_3}\}$ are all non-empty, then
they are $3$-cross $\tau$-intersecting,
and any two of them are $2$-cross $\tau$-intersecting.
\end{claim}

\begin{proof}
Let us start with two concrete examples.

First example is the case $I_1=I_2=\{1\}$, $I_3=\{2,3\}$, and so
$s=1+1+2=4$, $\tau=7-4=3$.
We show that $\{\GG_1,\GG_1,\GG_{23}\}$ are 3-cross 3-intersecting.
Suppose the contrary.
Then there are $G_1,G_1'\in\GG_1$, $G_{23}\in\GG_{23}$ and 
$x,y\in[4,n]$ such that $G_1\cap G_1'\cap G_{23}\subset\{x,y\}$.
Let $F_{12}=\{1,2\}\cup(G_1\setminus\{x\})$,
$F_{13}=\{1,3\}\cup(G_1'\setminus\{y\})$, and
$F_{23}=\{2,3\}\cup G_{23}\in\FF_{23}$.
By the shiftedness we have $F_{12}\in\FF_{12}$ and $F_{13}\in\FF_{13}$.
But $F_{12}\cap F_{13}\cap F_{23}=\emptyset$, a contradiction.

Next example is the case $I_1=I_2=I_3=\emptyset$, and so $\tau=7$.
We show that $\{\GG_\emptyset,\GG_\emptyset,\GG_\emptyset\}$ are 3-cross 
7-intersecting, that is, $\GG_\emptyset$ is 3-wise 7-intersecting.
Suppose the contrary.
Then there are $F,F',F''\in\FF_\emptyset$ and $x_1,\ldots,x_6\in[4,n]$ 
such that $F\cap F'\cap F''\subset\{x_1,\ldots,x_6\}$.
Let $F_{12}:=(F\setminus\{x_1,x_2\})\cup\{1,2\}\in\FF_{12}$, 
$F'_{13}:=(F'\setminus\{x_3,x_4\})\cup\{1,3\}\in\FF_{13}$, and
$F''_{23}:=(F'\setminus\{x_5,x_6\})\cup\{2,3\}\in\FF_{23}$.
Then we have $F_{12}\cap F'_{13}\cap F''_{23}=\emptyset$, a contradiction.

The following proof for the general case is given by one of the referees.
We assume by contradiction that there are sets $G_i\in\GG_I$ such that 
$|G_1\cap G_2 \cap G_| \leq \tau-1 = 6 - s$. 
The matrix $M$ has s ``taken'' places out of 9. ``Reserve'' one empty spot 
in each column (these are available since the matrix is acceptable). 
We are left with $6-s$ empty spots, say $r_i$ in row $i$. 
Shift $r_i$ elements from $G_i$ to the empty spots on row $i$ to construct 
a new set $F_i$, no longer belonging to $G_{I_i}$. 
By construction, $F_1, F_2, F_3$ have empty intersection.
\end{proof}

\subsection{Case $\GG_1=\emptyset$}
In this case, by Claim~\ref{poset}, we have 
$\GG_\emptyset=\GG_3=\GG_2=\GG_1=\emptyset$.

First suppose that $\GG_{23}\neq\emptyset$.
If $\bigcap G\neq\emptyset$, where the intersection is taken over all $G\in\GG_{12}\cup\GG_{13}\cup\GG_{23}$, then since the family is shifted,
$4\in\bigcap G$.
Since $\FF\subset\FF_{23}\cup\FF_{13}\cup\FF_{12}\cup\FF_{123}$, we have
$|F\cap[4]|\geq 3$ for every $F\in\FF$. 
This means that $\FF\subset\BD_3(n)$, which contradicts our assumption. 
Therefore we have $\bigcap G=\emptyset$. 
Moreover, the families $\GG_{12},\GG_{13},\GG_{23}$ are $3$-cross 
intersecting by Claim~\ref{M}. 
Thus we can apply \eqref{ineq<3p-e} of Theorem~\ref{thm3} with 
$\min \{\epsilon_p:\frac25\leq p\leq\frac12\}=\frac4{25}$ to get
\[
x_{12}+x_{13}+x_{23}\leq 3p-\epsilon_p\leq 3p-0.16.
\]
Next suppose that $\GG_{23}=\emptyset$. By Lemma~\ref{lemma:2c1i}
we have $x_{12}+x_{13}+x_{23}=x_{12}+x_{13}\leq 1\leq 3p-0.2$. 
Thus in both cases we have $x_{12}+x_{13}+x_{23}\leq 3p-0.16$. 
Then it follows from \eqref{mu(FF)} that
\begin{align*}
 \mu_p(\FF) &= p^2q(x_{12}+x_{13}+x_{23})+p^3x_{123}\\
&\leq p^2q(3p-0.16)+p^3\\
&=4p^3q+p^4-0.16p^2q.
\end{align*}
Noting that $\mu_p(\BD_3(n))=4p^3q+p^4$ and $p^2q\geq\frac{12}{125}=0.96$
for $\frac25\leq p\leq \frac12$, we have 
\[
\mu_p(\FF)\leq 4p^3q+p^4-0.16\cdot 0.96<\mu_p(\BD_3(n))-0.01, 
\]
as needed. 

\subsection{Case $\GG_1\neq\emptyset$ and $\GG_2=\emptyset$}
If $\GG_{23}=\emptyset$ then $[2,n]\not\in\FF$. This means that there are 
$F,F'\in\FF$ such that $F\cap F'=\{1\}$. (Otherwise all $F,F'\in\FF$ 
intersect on $[2,n]$ and we could add $[2,n]$ to $\FF$, which contradicts 
the assumption that $\FF$ is size maximal.)
In this case all subsets in $\FF$ must contain $1$, which contradicts the 
assumption that $\FF$ is non-trivial. So we may assume that 
$\GG_{23}\neq\emptyset$. Then both $\FF_1$ and $\FF_{23}$ are non-empty,
and so the families $\FF_{13}, \FF_{12}, \FF_{123}$ are also non-empty
by Claim~\ref{poset}.

By Claim~\ref{M} we have the following. 

\begin{claim}\label{claim:g1}
\begin{enumerate}
\item\label{g1-1} $\{\GG_1,\GG_1,\GG_{23}\}$ are $3$-cross $3$-intersecting,
and so $\GG_1$ is $2$-wise $3$-intersecting.
\item\label{g1-2} $\{\GG_{12},\GG_{13},\GG_{23}\}$ are $3$-cross intersecting, 
and so $\{\GG_{12},\GG_{13}\}$ are $2$-cross intersecting.
\item\label{g1-3} 
$\{\GG_1,\GG_{23},\GG_{123}\}$ are $3$-cross intersecting, and so
both $\{\GG_1,\GG_{123}\}$ and $\{\GG_{23},\GG_{123}\}$ are $2$-cross
intersecting.
\item\label{g1-4} $\{\GG_{1},\GG_{12},\GG_{23}\}$ are $3$-cross 
$2$-intersecting.
\item\label{g1-5} 
$\{\GG_{1},\GG_{23},\GG_{23}\}$ are $3$-cross $2$-intersecting, and so
$\GG_{23}$ is $2$-wise $2$-intersecting.
\end{enumerate}
\end{claim}

\begin{claim}\label{claim:g1eqn}
\begin{enumerate}
\item\label{g1eqn-1} $\min\{x_1,x_{23}\}\leq \tilde\alpha^3$
(see Claim~\ref{tilde-alpha} for the definition of $\tilde\alpha$).
\item\label{g1eqn-2} $x_{12}+c x_{13}\leq p(c+1)$, where
$c=\frac1{2p}(1+\sqrt{1-4p^2})$.
\item\label{g1eqn-3} $x_1+x_{123}\leq 1$.
\item\label{g1eqn-4} $x_{23}+x_{123}\leq 1$.
\item\label{g1eqn-5} $x_{1}+x_{12}+x_{23} \leq 1$.
\item\label{g1eqn-6} 
$x_{23}\leq \tilde a_2$
(see Claim~\ref{tilde-at} for the definition of $\tilde a_t$).
\item\label{g1eqn-7} $x_1\leq \tilde a_3$.
\end{enumerate} 
\end{claim}

\begin{proof}
Item \eqref{g1eqn-1}: 
By Lemma~\ref{la2} with \eqref{g1-1} of Claim~\ref{claim:g1}, we have
$x_1^2x_{23}\leq\alpha^9$. Then, using $\alpha\leq\tilde\alpha$ from
Claim~\ref{tilde-alpha}, we get 
$(\min\{x_1,x_{23}\})^3\leq x_1^2x_{23}\leq\alpha^9\leq\tilde\alpha^9$.

Item \eqref{g1eqn-2}:
By (i) of Lemma~\ref{lemma:2c1i} with \eqref{g1-2} of Claim~\ref{claim:g1}, 
we have $x_{13}+x_{12}\leq 1$. 
Moreover, by (ii) of Lemma~\ref{lemma:2c1i} with $x_{13}\leq x_{12}$
from Claim~\ref{poset}, 
we have $x_{13}^2\leq x_{13}x_{12}\leq p^2$ and $x_{13}\leq p$. 
By solving $x_{13}x_{12}=p^2$ and $x_{13}+x_{12}=1$ 
with $x_{13}\leq x_{12}$ we get 
$(x_{13},x_{12})=(\frac12(1-\sqrt{1-4p^2}),\frac12(1+\sqrt{1-4p^2}))
=(1-cp,cp)$.
Also, by solving $x_{13}x_{12}=p^2$ and $x_{13}=x_{12}$ we get 
$(x_{13},x_{12})=(p,p)$.
Thus $(x_{13},x_{12})$ exists only under the line connecting these two
points, that is, $x_{12}\leq c(p-x_{13})+p$. (See Figure~\ref{fig-claim20})
\begin{figure}[h]
\includegraphics[scale=1.2]{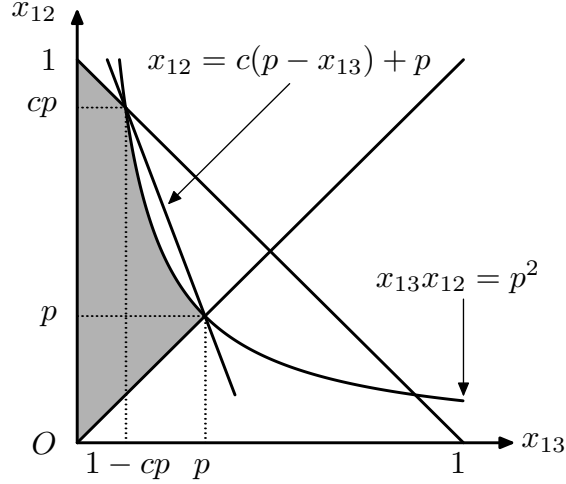}  
\caption{Item (2) of Claim~20: $(x_{13},x_{12})$ is included in the gray area}\label{fig-claim20}
\end{figure}

Items \eqref{g1eqn-3} and \eqref{g1eqn-4}: These follow from 
(i) of Lemma~\ref{lemma:2c1i} with \eqref{g1-3} of Claim~\ref{claim:g1}. 

Items \eqref{g1eqn-5}: This follows from 
Lemma~\ref{lemma:3w1i<=1} with \eqref{g1-4} of Claim~\ref{claim:g1}. 

Items \eqref{g1eqn-6}: This follows from Claim~\ref{tilde-at} 
with \eqref{g1-5} of Claim~\ref{claim:g1}. 

Items \eqref{g1eqn-7}: This follows from Claim~\ref{tilde-at} 
with \eqref{g1-1} of Claim~\ref{claim:g1}. 
\end{proof}

Recall from \eqref{mu(FF)} that 
$\mu_p(\FF)=pq^2x_1+p^2q(x_{12}+x_{13}+x_{23})+p^3x_{123}$. 

\subsubsection{Subcase $x_1\leq x_{23}$.}
We solve the following linear programming problem:
\begin{description}
\item[maximize] $pq^2x_1+p^2q(x_{12}+x_{13}+x_{23})+p^3x_{123}$,
\item[subject to] $x_1-x_{23}\leq 0$, \eqref{g1eqn-1}--\eqref{g1eqn-6}
in Claim~\ref{claim:g1eqn}, and $x_I\geq 0$ for all $I$.
\end{description}
The corresponding dual problem is
\begin{description}
\item[minimize] $\tilde\alpha^3 y_1+p(c+1) y_2+y_3+y_4+y_5+\tilde a_2 y_6$,
\item[subject to]
$y_0+y_1+y_3+y_5\geq pq^2$, 
$-y_0+y_4+y_5+y_6\geq p^2q$, 
$c y_2\geq p^2q$, 
$y_2+y_5\geq p^2q$, 
$y_3+y_4\geq p^3$, and $y_i\geq 0$ for all $i$. 
\end{description}
\begin{table}[h]
\caption{Subcase $x_1\leq x_{23}$}
\begin{tabular}{c|ccccc|c}
&$x_1$&$x_{23}$&$x_{13}$&$x_{23}$&$x_{123}$&\\
\hline
$y_0$&1&$-1$&&&&$0$\\
$y_1$&1&&&&&$\tilde\alpha^3$\\
$y_2$&&&$c$&1&&$p(c+1)$\\
$y_3$&1&&&&1&$1$\\
$y_4$&&1&&&1&$1$\\
$y_5$&1&1&&1&&$1$\\
$y_6$&&1&&&&$\tilde a_2$\\
\hline
&$pq^2$&$p^2q$&$p^2q$&$p^2q$&$p^3$
\end{tabular}
\end{table}
A feasible solution is given by
$y_0=y_6=0$,
$y_1=pq^2-p^2q(1-\frac2c)-p^3$,
$y_2=y_4=\frac{p^2q}c$,
$y_3=p^3-\frac{p^2q}c$,
$y_5=p^2q(1-\frac1c)$,
and the corresponding value of the objective function is
\begin{align}\label{value4.2.1}
p\left(p + p^2 - p^3 + \tilde\alpha^3(1 - 3p + p^2)\right) 
-\frac1c \,p^2 q(1-2\tilde\alpha^3-p).
\end{align}
By the Taylor expansion of $\frac1c$ at $p=\frac25$ it follows that
$\frac1c>d(p)$, where 
$d(p)=\frac{1375 p^2}{216}-\frac{325p}{108}+\frac{37}{54}$ for 
$\frac25\leq p\leq \frac12$. Since $1-2\tilde\alpha^3-p>0$ in this 
interval, the value \eqref{value4.2.1} satisfies
\begin{align*}
 &<
p\left(p + p^2 - p^3 + \tilde\alpha^3(1 - 3p + p^2)\right) 
-d(p)p^2 q(1-2\tilde\alpha^3-p) =:f(p)\\
&\approx  - 0.00633167 p + 0.490037 p^2 + 3.71975 p^3 - 8.76595 p^4 \\
&\qquad\quad +21.3084 p^5 - 61.7755 p^6 + 95.9036 p^7 - 52.7438 p^8.
\end{align*}
Let $g(p):=(4p^3q+p^4 -0.00194)-f(p)$, and let $g^{(i)}(p)$ denote the
$i$-th derivative. Let $\frac25\leq p\leq \frac12$.
We have $g^{(6)}(p)\approx 1063314.937p^2-483354.3468p+44478.39041$
and $g^{(6)}(p)>0$. Thus $g^{(4)}(p)$ is convex.
Since $g^{(4)}(\frac25)<-213<0$ and $g^{(4)}(\frac12)<-112<0$, we have
$g^{(4)}(p)<0$. Thus $g^{(3)}(p)$ is decreasing in $p$. Since 
$g^{(3)}(\frac25)<-7<0$ we have $g^{(3)}(p)<0$. 
Thus $g'(p)$ is concave. Since $g'(\frac25)>0.23>0$ and 
$g'(\frac12)>0.34>0$, we have $g'(p)>0$ and $g(p)$ is increasing in $p$.
Finally we have $g(\frac25)>3\times 10^{-6}>0$, and so $g(p)>0$. This means
that the value \eqref{value4.2.1} is less than $4p^3q+p^4 -0.00194$.

Then by the weak duality theorem we have
$\mu_p(\FF)<4p^3q+p^4-0.00194$.

\subsubsection{Subcase $x_{23}\leq x_1$.}
We solve the following linear programming problem:
\begin{description}
\item[maximize] $pq^2x_1+p^2q(x_{12}+x_{13}+x_{23})+p^3x_{123}$,
\item[subject to] $x_{23}-x_1\leq 0$,
\eqref{g1eqn-1}--\eqref{g1eqn-5} and \eqref{g1eqn-7}
in Claim~\ref{claim:g1eqn}, and $x_I\geq 0$ for all $I$.
\end{description}
The corresponding dual problem is
\begin{description}
\item[minimize] 
$\tilde\alpha^3 y_1+p(c+1) y_2+ y_3+y_4+y_5+\tilde a_3 y_7$,
\item[subject to]
$-y_0+y_3+y_5+y_7\geq pq^2$, 
$y_0+y_1+y_4+y_5\geq p^2q$, 
$c y_2\geq p^2q$, 
$y_2+y_5\geq p^2q$, 
$y_3+y_4\geq p^3$, and $y_i\geq 0$ for all $i$. 
\end{description}
\begin{table}[h]
\caption{Subcase $x_{23}\leq x_1$}
\begin{tabular}{c|ccccc|c}
&$x_1$&$x_{23}$&$x_{13}$&$x_{23}$&$x_{123}$&\\
\hline
$y_0$&$-1$&1&&&&$0$\\
$y_1$&&1&&&&$\tilde\alpha^3$\\
$y_2$&&&$c$&1&&$p(c+1)$\\
$y_3$&1&&&&1&$1$\\
$y_4$&&1&&&1&$1$\\
$y_5$&1&1&&1&&$1$\\
$y_7$&1&&&&&$\tilde a_3$\\
\hline
&$pq^2$&$p^2q$&$p^2q$&$p^2q$&$p^3$
\end{tabular}
\end{table}
A feasible solution is given by
$y_0=y_4=0$,
$y_1=y_2=p^2 q(1-\frac1c)$,
$y_3=p^3$,
$y_5=p^2 q(1-\frac1c)$,
$y_7=p q^2-p^2 q(1-\frac1c)-p^3$.
Noting that $c+\frac1c=p$, 
the corresponding value of the objective function is
\[
p^3\left(-\tilde\alpha^3+\tilde a_3-1\right)+p^2 \left(\tilde\alpha^3-3\tilde a_3+2\right)+\tilde a_3 p
   -\frac1c\,p^2q\left(\tilde\alpha^3-\tilde a_3+2 p+1\right).
\]
Then, using $\tilde\alpha^3-\tilde a_3+2 p+1>0$ and  $\frac1c>d(p)$ 
(see the previous subsection), the above value satisfies
\begin{align*}
&< -1.50324 p + 8.79901 p^2 - 3.86493 p^3 - 26.8212 p^4\\
&\qquad\quad  + 30.6062 p^5 + 12.8608 p^6 - 47.9518 p^7 + 26.3719 p^8=:f(p).
\end{align*}
Let $g(p):=(4p^3q+p^4 -0.00182)-f(p)$.
Let $J_1=\{p\in\R:\frac25\leq p\leq\frac9{20}\}$,
$J_2=\{p\in\R:\frac9{20}\leq p\leq\frac12\}$, and $J=J_1\cup J_2$.
We need to show that $g(p)>0$ for $p\in J$.

First we show that $g^{(4)}(p)>0$ for $p\in J_1$. 
We have $g^{(7)}(p)<0$ for $p\in J$. Thus $g^{(5)}(p)$ is concave.
Since $g^{(5)}(\frac25)>0$ and $g^{(5)}(\frac9{20})>0$ we have
$g^{(5)}(p)>0$ for $p\in J_1$. Thus $g^4(p)$ is increasing in $p\in J_1$.
Since $g^{(4)}(\frac9{20})<0$, we have $g^{(4)}(p)<0$ for $p\in J_1$.

Next we show that $g^{(4)}(p)>0$ for $p\in J_2$.
Since $g^{(7)}(p)<0$, $g^{(6)}(p)$ is decreasing in $p$.
Since $g^{(6)}(\frac9{20})<0$, $g^{(4)}(p)$ is concave for $p\in J_2$.
Let $h(p)=481.064p-381.36$ be the tangent to $g^{(4)}(p)$ at $p=\frac9{20}$.
Then we have $g^{(4)}(p)\leq h(p)\leq h(\frac12)$ for $p\in J_2$. 
Since $h(\frac12)<0$, we have $g^{(4)}(p)<0$ for $p\in J_2$.

Let $p\in J$. We have shown that $g^{(4)}(p)>0$.
Then $g^{(2)}(p)$ is concave. Since $g^{(2)}(\frac25)>0$
and $g^{(2)}(\frac12)>0$, we have $g^{(2)}(p)>0$. Thus $g'(p)$ is increasing
in $p$. Since $g'(\frac25)>0$, we have $g'(p)>0$ and $g(p)$ is increasing
in $p$. Since $g(\frac25)>0$, we have $g(p)>0$ as needed.

Thus it follows from the weak duality theorem that 
$\mu_p(\FF)<4p^3q+p^4-0.0018$.

\subsection{Case $\GG_2\neq\emptyset$, $\GG_3=\emptyset$}
Using Claim~\ref{M} we have the following.
\begin{claim}\label{claim:g2}
\begin{enumerate}
\item\label{g21} 
$\{\GG_1,\GG_1,\GG_2\}$ are $3$-cross $4$-intersecting, and so $\GG_1$ is
$2$-wise $4$-intersecting.
\item\label{g22} $\{\GG_2,\GG_{13},\GG_{13}\}$ are $3$-cross $2$-intersecting,
and so $\GG_{13}$ is $2$-wise $2$-intersecting.
\item\label{g23} 
$\{\GG_2,\GG_{13},\GG_{123}\}$ are $3$-cross intersecting, and so
$\{\GG_{13},\GG_{123}\}$ are $2$-cross intersecting.
\item\label{g24} $\{\GG_{2},\GG_{12},\GG_{13}\}$ are $3$-cross $2$-intersecting.
\item\label{g25} $\{\GG_{1},\GG_{2},\GG_{123}\}$ are $3$-cross $2$-intersecting.
\end{enumerate}
\end{claim}

\begin{claim}\label{claim:g2eqn}
\begin{enumerate}
\item\label{g2eqn-1} $x_1\leq \tilde a_4$. 
\item\label{g2eqn-2} $x_{13}\leq \tilde a_2$.
\item\label{g2eqn-3} $x_2\leq \tilde \alpha^4$. 
\item\label{g2eqn-4} $x_{13}+x_{123}\leq 1$.
\item\label{g2eqn-5} $x_{1}+x_{12}+x_{23}\leq 1$.
\item\label{g2eqn-6} $x_{2}+x_{12}+x_{13}\leq 1$.
\item\label{g2eqn-7} $x_{1}+x_{2}+x_{123}\leq 1$.
\end{enumerate}
\end{claim}
\begin{proof}
Item \eqref{g2eqn-3}: 
By Lemma~\ref{la2} with \eqref{g21} of Claim~\ref{claim:g2}, we have
$x_1^2x_2\leq\alpha^{12}$. Then, using $x_2\leq x_1$ from Claim~\ref{poset}
and Claim~\ref{tilde-alpha}, we get 
$x_2^3\leq x_1^2x_2\leq\alpha^{12}\leq\tilde\alpha^{12}$.

Item \eqref{g2eqn-5} is from Claim~\ref{claim:g1eqn}. Indeed all 
inequalities in Claim~\ref{claim:g1eqn} are still valid in this case. 

The other items follow from Claim~\ref{claim:g2}, Claim~\ref{tilde-at},
Lemma~\ref{lemma:2c1i}, and Lemma~\ref{lemma:3w1i<=1}.
\end{proof}

We solve the following linear programming problem:
\begin{description}
\item[maximize] $pq^2(x_1+x_2)+p^2q(x_{12}+x_{13}+x_{23})+p^3x_{123}$,
\item[subject to] \eqref{g2eqn-1}--\eqref{g2eqn-7}
in Claim~\ref{claim:g2eqn}, and $x_I\geq 0$ for all $I$.
\end{description}
The corresponding dual problem is
\begin{description}
\item[minimize] 
$\tilde a_4y_1+\tilde a_2y_2+\tilde\alpha^4 y_3+ y_4+ y_5+y_6+y_7$,
\item[subject to]
$y_3+y_6+y_7\geq pq^2$, 
$y_1+y_5+y_7\geq pq^2$, 
$y_5\geq p^2q$, 
$y_2+y_4+y_6\geq p^2q$, 
$y_5+y_6\geq p^2q$, 
$y_4+y_7\geq p^3$, and $y_i\geq 0$ for all $i$. 
\end{description}
\begin{table}[h]
\caption{Case $\GG_2\neq\emptyset$}
\begin{tabular}{c|cccccc|c}
&$x_2$&$x_1$&$x_{23}$&$x_{13}$&$x_{23}$&$x_{123}$&\\
\hline
$y_1$&&1&&&&&$\tilde a_4$\\
$y_2$&&&&1&&&$\tilde a_2$\\
$y_3$&1&&&&&&$\tilde\alpha^4$\\
$y_4$&&&&1&&1&$1$\\
$y_5$&&1&1&&1&&$1$\\
$y_6$&1&&&1&1&&$1$\\
$y_7$&1&1&&&&1&$1$\\
\hline
&$pq^2$&$pq^2$&$p^2q$&$p^2q$&$p^2q$&$p^3$
\end{tabular}
\end{table}
A feasible solution is given by
$y_1=pq(1-2p)$, $y_2=p^2(1-2p)$, 
$y_3=pq^2$, $y_4=p^3$,
$y_5=p^2 q$, $y_6=y_7=0$.
Then the corresponding value of the objective function is
\begin{align*}
& \tilde a_4pq(1-2p) +  \tilde a_2p^2(1-2p) +\tilde\alpha^4 pq^2+p^2\\
&\qquad\approx -1.70723 p + 9.39501 p^2 - 10.639 p^3 - 1.74811 p^4 \\
&\qquad\qquad + 13.3146 p^5 -  16.3729 p^6 + 6.6536 p^7\\
&\qquad< 4p^3q+p^4-0.00436,
\end{align*}
and so $\mu_p(\FF)<4p^3q+p^4-0.004$.

\subsection{Case $\GG_3\neq\emptyset$, $\GG_\emptyset=\emptyset$}
Using Claim~\ref{M} we have the following.
\begin{claim}\label{claim:g3}
\begin{enumerate}
\item\label{g31} 
$\{\GG_3,\GG_{12},\GG_{12}\}$ are $3$-cross $2$-intersecting, and so
$\GG_{12}$ is $2$-wise $2$-intersecting.
\item\label{g32} 
$\{\GG_3,\GG_{12},\GG_{123}\}$ are $3$-cross intersecting, and so
$\{\GG_{12},\GG_{123}\}$ are $2$-cross intersecting.
\end{enumerate}
\end{claim}

\begin{claim}\label{claim:g3eqn}
\begin{enumerate}
\item\label{g3eqn-1} $x_1\leq \tilde a_4$. 
\item\label{g3eqn-2} $x_{12}\leq \tilde a_2$.
\item\label{g3eqn-3} $x_2\leq \tilde \alpha^4$. 
\item\label{g3eqn-4} $x_{12}+x_{123}\leq 1$.
\item\label{g3eqn-5} $x_{1}+x_{12}+x_{23}\leq 1$.
\item\label{g3eqn-6} $x_{2}+x_{12}+x_{13}\leq 1$.
\item\label{g3eqn-7} $x_{1}+x_{2}+x_{123}\leq 1$.
\item\label{g3eqn-8} $x_3-x_2\leq 0$.
\item\label{g3eqn-9} $x_{23}-x_{13}\leq 0$.
\item\label{g3eqn-10} $x_{13}-x_{12}\leq 0$.
\end{enumerate}
\end{claim}
\begin{proof}
Items \eqref{g3eqn-1}, \eqref{g3eqn-3}, \eqref{g3eqn-6}, and 
\eqref{g3eqn-7} are from Claim~\ref{claim:g2}.
Items \eqref{g3eqn-2} and \eqref{g3eqn-4} follow from Claim~\ref{claim:g3},
Claim~\ref{tilde-at}, and Lemma~\ref{lemma:2c1i}.
Item \eqref{g3eqn-5} is from Claim~\ref{claim:g1}.
The other items are from Claim~\ref{poset}.
\end{proof}

We solve the following linear programming problem:
\begin{description}
\item[maximize] $pq^2(x_1+x_2+x_3)+p^2q(x_{12}+x_{13}+x_{23})+p^3x_{123}$,
\item[subject to] \eqref{g3eqn-1}--\eqref{g3eqn-10}
in Claim~\ref{claim:g3eqn}, and $x_I\geq 0$ for all $I$.
\end{description}
The corresponding dual problem is
\begin{description}
\item[minimize] 
$\tilde a_4y_1+\tilde a_2y_2+\tilde\alpha^4 y_3+ y_4+ y_5+y_6+y_7$,
\item[subject to]
$y_8\geq pq^2$,
$y_3+y_6+y_7-y_8\geq pq^2$, 
$y_1+y_5+y_7\geq pq^2$, 
$y_5+y_9\geq p^2q$, 
$y_6-y_9+y_{10}\geq p^2q$, 
$y_2+y_4+y_5+y_6-y_{10}\geq p^2q$, 
$y_4+y_7\geq p^3$, and $y_i\geq 0$ for all $i$. 
\end{description}
\begin{table}[h]
\caption{Case $\GG_3\neq\emptyset$}
\begin{tabular}{c|ccccccc|c}
&$x_3$&$x_2$&$x_1$&$x_{23}$&$x_{13}$&$x_{23}$&$x_{123}$&\\
\hline
$y_1$&&&1&&&&&$\tilde a_4$\\
$y_2$&&&&&&1&&$\tilde a_2$\\
$y_3$&&1&&&&&&$\tilde\alpha^4$\\
$y_4$&&&&&&1&1&$1$\\
$y_5$&&&1&1&&1&&$1$\\
$y_6$&&1&&&1&1&&$1$\\
$y_7$&&1&1&&&&1&$1$\\
$y_8$&1&$-1$&&&&&&$0$\\
$y_9$&&&&$1$&$-1$&&&$0$\\
$y_{10}$&&&&&$1$&$-1$&&$0$\\
\hline
&$pq^2$&$pq^2$&$pq^2$&$p^2q$&$p^2q$&$p^2q$&$p^3$
\end{tabular}
\end{table}
We distinguish the following two subcases.

\subsubsection{Subcase $\frac25\leq p\leq 0.453264$}
A feasible solution is given by
$y_1=p q^2$,
$y_2=p^2(3-4p)$,
$y_3=2pq^2$,
$y_4=p^3$,
$y_5=y_6=y_7=0$,
$y_8=pq^2$,
$y_9=p^2q$,
$y_{10}=2p^2q$,
and the corresponding value of the objective function is
\begin{align*}
&
p\left(\tilde a_4q^2+\tilde a_2(3-4p)p+2\tilde\alpha^4q^2+p^2\right)\\
&\qquad\approx 
-1.70606 p + 4.43558 p^2 + 5.72241 p^3 - 16.7467 p^4\\
&\qquad\quad + 26.6293 p^5 - 32.7457 p^6 + 13.3072 p^7\\
&\qquad<4p^3q+p^4-0.00404377.
\end{align*}
Thus $\mu_p(\FF)<4p^3q+p^4-0.004$.

\subsubsection{Subcase $0.453264\leq p\leq \frac 12$}
A feasible solution is given by
$y_1=y_6=y_9=0$,
$y_2=p(1-2p)$,
$y_3=pq$,
$y_4=p(3p-p^2-1)$,
$y_5=y_{10}=p^2q$,
$y_7=pq(1-2p)$,
$y_8=pq^2$,
and the corresponding value of the objective function is
\begin{align*}
&
p\left(\tilde a_2(1-2p)+\tilde\alpha^4q+p\right)\\
&\qquad\approx 
-1.10283 p + 6.37415 p^2 - 5.84563 p^3 
- 3.59536 p^4 + 9.71927 p^5 - 6.6536 p^6\\
&\qquad
<4p^3q+p^4-0.00404377.
\end{align*}
Thus $\mu_p(\FF)<4p^3q+p^4-0.004$.

\subsection{Case $\GG_\emptyset\neq\emptyset$}
Using Claim~\ref{M} we have the following.
\begin{claim}\label{claim:gempty}
\begin{enumerate}
\item\label{g01} $\{\GG_\emptyset,\GG_\emptyset,\GG_\emptyset\}$ are
$3$-cross $7$-intersecting, that is,
$\GG_\emptyset$ is $3$-wise $7$-intersecting.
\item\label{g02} 
$\{\GG_1,\GG_1,\GG_\emptyset\}$ are $3$-cross $5$-intersecting,
and so $\GG_1$ is $2$-wise $5$-intersecting.
\item\label{g03} 
$\{\GG_\emptyset,\GG_{123},\GG_{123}\}$ are $3$-cross intersecting, and so
$\{\GG_{123},\GG_{123}\}$ are $2$-cross intersecting.
\end{enumerate}
\end{claim}

\begin{claim}\label{claim:g0eqn}
\begin{enumerate}
\item\label{g0eqn-1} $x_\emptyset\leq \tilde\alpha^7$.
\item\label{g0eqn-2} $x_2\leq \tilde\alpha^4$.
\item\label{g0eqn-3}\label{g0x1} $x_1\leq \tilde a_5$.
\item\label{g0eqn-4} $x_{123}\leq p$.
\item\label{g0eqn-5} $x_{1}+x_{12}+x_{23}\leq 1$.
\item\label{g0eqn-6} $x_3-x_2\leq 0$.
\item\label{g0eqn-7} $x_{23}-x_{13}\leq 0$.
\item\label{g0eqn-8} $x_{13}-x_{12}\leq 0$.
\item\label{g0eqn-9} $x_{12}-x_{123}\leq 0$.
\end{enumerate}
\end{claim}

\begin{proof}
Items \eqref{g0eqn-1}, \eqref{g0eqn-3}, and \eqref{g0eqn-4} follow from
Claim~\ref{claim:gempty}, Lemma~\ref{la2}, Claim~\ref{tilde-alpha},
Claim~\ref{tilde-at}, and Lemma~\ref{lemma:2c1i}.
Items \eqref{g0eqn-2} and \eqref{g0eqn-5} are from Claim~\ref{claim:g2eqn}
and Claim~\ref{claim:g1eqn}, respectively.
The other items are from Claim~\ref{poset}.
\end{proof}

We solve the following linear programming problem:
\begin{description}
\item[maximize] 
$q^3 x_\emptyset+pq^2(x_1+x_2+x_3)+p^2q(x_{12}+x_{13}+x_{23})+p^3x_{123}$,
\item[subject to] \eqref{g0eqn-1}--\eqref{g0eqn-9}
in Claim~\ref{claim:g0eqn}, and $x_I\geq 0$ for all $I$.
\end{description}
The corresponding dual problem is
\begin{description}
\item[minimize] 
$\tilde\alpha^7 y_1+\tilde\alpha^4 y_2+\tilde a_5 y_3+ p y_4+ y_5$,
\item[subject to]
$y_1\geq q^3$,
$y_6\geq p q^2$,
$y_2-y_6\geq p q^2$,
$y_3+y_5\geq p q^2$,
$y_5+y_7\geq p^2 q$,
$-y_7+y_8\geq p^2 q$,
$y_5-y_8+y_9\geq p^2 q$,
$y_4-y_9\geq p^3$,
and $y_i\geq 0$ for all $i$. 
\end{description}
\begin{table}[h]
\caption{Case $\GG_\emptyset\neq\emptyset$}
\begin{tabular}{c|cccccccc|c}
&$x_\emptyset$&$x_3$&$x_2$&$x_1$&$x_{23}$&$x_{13}$&$x_{23}$&$x_{123}$&\\
\hline
$y_1$&1&&&&&&&&$\tilde\alpha^7$\\
$y_2$&&&1&&&&&&$\tilde\alpha^4$\\
$y_3$&&&&1&&&&&$\tilde\alpha^3$\\
$y_4$&&&&&&&&1&$p$\\
$y_5$&&&&1&1&&1&&$1$\\
$y_6$&&1&$-1$&&&&&&$0$\\
$y_7$&&&&&$1$&$-1$&&&$0$\\
$y_8$&&&&&&$1$&$-1$&&$0$\\
$y_9$&&&&&&&$1$&$-1$&$0$\\
\hline
&$q^3$&$pq^2$&$pq^2$&$pq^2$&$p^2q$&$p^2q$&$p^2q$&$p^3$
\end{tabular}
\end{table}
We distinguish the following two subcases.

\subsubsection{Subcase $\frac25\leq p\leq 0.424803$}
A feasible solution is given by
$y_1=q^3$,
$y_2=2pq^2$,
$y_3=y_6=pq^2$,
$y_4=p^2(3-2p)$,
$y_5=0$,
$y_7=p^2q$,
$y_8=2p^2 q$,
$y_9=3p^2 q$,
and the corresponding value of the objective function is
\begin{align*}
&\tilde\alpha^7 q^3 + 2\tilde\alpha^4 pq^2+ 
\tilde a_5 pq^2+p^3(3-2p)\\
&\qquad\approx 
-27.5644 p^{10}+104.919p^9-157.051 p^8+132.065p^7-82.6061 p^6+39.2544p^5\\
&\qquad\qquad -7.70344 p^4-6.68845p^3+8.19629 p^2-1.82104p-7.41682\cdot 10^{-6}\\
&\qquad<4p^3q+p^4-0.004322.
\end{align*}
Thus $\mu_p(\FF)<4p^3q+p^4-0.004$.

\subsubsection{Subcase $0.424803\leq p\leq\frac12$}
A feasible solution is given by
$y_1=q^3$,
$y_2=2pq^2$,
$y_3=pq(1-2p)$,
$y_4=p^2$,
$y_5=y_8=y_9=p^2q$,
$y_6=pq^2$,
$y_7=0$,
and the corresponding value of the objective function is
\begin{align*}
&\tilde\alpha^7 q^3 + 2\tilde\alpha^4 pq^2
+ \tilde a_5 pq(1-2p)+p^2\\
&\qquad\approx 
-27.5644 p^{10}+104.919p^9-157.051 p^8+132.065p^7-82.6061 p^6+39.2544p^5\\
&\qquad\qquad
-1.05572 p^4-16.16p^3+11.0202 p^2-1.82104p-7.41682\cdot 10^{-6}\\
&\qquad<4p^3q+p^4-0.004322.
\end{align*}
Thus $\mu_p(\FF)<4p^3q+p^4-0.004$.

\medskip
This completes the proof of Theorem~\ref{thm:main}. \qed

\section{Concluding remarks}

In this section we discuss possible extensions and related problems.
\subsection{Non-trivial $r$-wise intersecting families for $r\geq 4$}
We have determined $M_2(p)$ and $M_3(p)$ for all $p$. 
Let us consider $M_r(p)$ for the general case $r\geq 4$.
Some of the facts we used for the cases $r=2,3$ can be easily extended 
for the other cases as follows.
\begin{prop}\label{prop:M_r(p)}
Let $r\geq 2$.
\begin{enumerate}
 \item For $s=0,1,\ldots,r-1$ we have 
$M_r(p)\geq p^s$ for $p>\frac{r-s-1}{r-s}$.
\item $M_r(p)=p^{r-1}$ for $0<p\leq\frac1r$.
\item $M_r(p)=p$ for $\frac{r-2}{r-1}<p\leq\frac{r-1}r$.
\item $M_r(p)=1$ for $\frac{r-1}{r}<p<1$.
\end{enumerate} 
\end{prop}
\begin{proof}
Item (1): We construct a non-trivial $r$-wise intersecting family
\[
\FF_r(n,s):=\{\{[s]\cup G:G\subset[s+1,n],\,
|G|\geq\tfrac{r-s-1}{r-s}n\}\}\cup\{[n]\setminus\{i\}:1\leq i\leq s\}.
\]
Then $\mu_p(\FF_r(n,s))\to p^s$ as $n\to\infty$, cf.~\cite{FT2006}.

Item (2): By item (1) with $s=r-1$ we have $M_r(p)\geq p^{r-1}$.
On the other hand, a non-trivial $r$-wise intersecting family is 
$2$-wise $(r-1)$-intersecting, and by Theorem~\ref{thm:AK} the $p$-measure
of the family is at most $p^{r-1}$ if $p\leq\frac1r$.

Item (3): By item (1) with $s=1$ we have $M_r(p)\geq p$.
On the other hand, it is known from \cite{FGL,FT2003,T2021} that
$r$-wise intersecting family has $p$-measure at most $p$ if 
$p\leq\frac{r-1}r$.

Item (4): By item (1) with $s=0$ we have $M_r(p)\geq 1$, and so 
$M_r(p)=1$ by definition of $M_r(p)$.
\end{proof}

Even for the case $r=4$ the exact value of $M_4(p)$ is not known for
$\frac14<p\leq\frac23$. In this case Proposition~\ref{prop:M_r(p)} and 
Theorem~\ref{thm:BD} give us the following.
For simplicity here we write 
$\bd_r(p)=\lim_{n\to\infty}\mu_p(\BD_r(n))$ and 
$f_r(p,s)=\lim_{n\to\infty}\mu_p(\FF_r(n,s))$.

\begin{fact}\label{fact 4-wise}
For non-trivial $4$-wise intersecting families we have the following:
\[
M_4(p)\begin{cases}
=f_4(p,3)=
p^3 & \text{ if }0<p\leq\frac14,\\
\geq \bd_4(p) = 5p^4q+p^5
& \text{ if }\frac14\leq p<\frac12,\\
=\bd_4(p) = 5p^4q+p^5 & \text{ if }p=\frac12,\\
\geq f_4(p,2)=
p^2 &\text{ if }\frac12<p\leq\frac{1+\sqrt{17}}8,\\
\geq \bd_4(p) = 5p^4q+p^5
&\text{ if }\frac{1+\sqrt{17}}8<p\leq\frac23,\\
=f_4(p,1)=p &\text{ if }\frac23<p\leq\frac34,\\
=1 &\text{ if }\frac34<p<1.
\end{cases}
\]
\end{fact}

\begin{conj}
For $r\geq 2$ it holds that
$M_r(p)=\bd_r(p)$ for $\frac1r\leq p\leq\frac12$.
\end{conj}

It is known that $M_r(p)=\bd_r(p)$ if $r\geq 13$ and
$\frac12\leq p\leq \frac12+\epsilon_r$ for some $\epsilon_r>0$, see
\cite{FT2006}. 
Note also that $M_5(p)\geq f_5(p,3)>\bd_5(p)$ for 
$\frac12<p<\frac{1+\sqrt{21}}{10}$.

\begin{prob}
Let $0<p\leq\frac{r-1}r$, and  
let $\FF\subset 2^{[n]}$ be a non-trivial $r$-wise intersecting family.
Is it true that 
\[
M_r(p)\leq\max\{\bd_r(p),\,
f_r(p,1),\ldots, f_r(p,r-1)\}\,?
\]
\end{prob}

\subsection{Uniform families}
One can consider non-trivial $r$-wise intersecting $k$-uniform families,
that is, families in $\binom{[n]}k:=\{F\subset[n]:|F|=k\}$, and ask the 
maximum size. Let us construct some candidate families to address this 
problem. For $1\leq s\leq r-1$ and $r-s+1\leq y\leq k-s+1$,
let $j_0:=\lceil\frac{(r-s-1)y+1}{r-s}\rceil$.
Note that $j_0<y$ and $j_0$ is the minimum integer $j$ satisfying 
$(r-s)j\geq(r-s-1)y+1$.
Let $\FF_r(n,k,s,y):=\AA\cup\BB$, where
\begin{align*}
\AA&:=\{A\in\tbinom{[n]}k:[s]\subset A,\,|A\cap[s+1,s+y]|\geq j_0\},\\
\BB&:=\{B\in\tbinom{[n]}k:|B\cap[s]|=s-1,\,[s+1,s+y]\subset B\}.
\end{align*}
Then the family $\{A\setminus[s]:A\in\AA\}$ is $(r-s)$-wise intersecting 
due to the choice of $j_0$ and $y$.
Thus $\FF_r(n,k,s,y)$ is a $k$-uniform non-trivial $r$-wise intersecting
family. In particular, 
\[
\FF_r(n,k,1,r)=\BD_r(n)\cap\tbinom{[n]}k\quad (j_0=r-1), 
\]
and
$\FF_r(n,k,r-1,k-r+2)$ ($j_0=1$) is the so-called Hilton--Milner family.
Note that different parameters may give the same family, e.g.,
$\FF_r(n,k,1,r)=\FF_r(n,k,s,r-s+1)$ for all $1\leq s\leq r-1$.
\begin{conj}[O'Neill and Verstr\"aete \cite{OV}]\label{ON-V conj}
Let $k>r\geq 2$ and $n\geq kr/(r-1)$. Then the unique extremal non-trivial 
$r$-wise intersecting families in $\binom{[n]}k$ are 
$\FF_r(n,k,1,r)$ and $\FF_r(n,k,r-1,k-r+2)$ (up to isomorphism).
\end{conj}
O'Neill and Verstr\"aete proved the conjecture if 
$n\geq r+e(k^22^k)^{2^k}(k-r)$. This bound can be reduced to
$n>(1+\frac r2)(k-r+2)$ using 
the Ahlswede--Khachatrian
theorem for non-trivial 2-wise $t$-intersecting families in \cite{AK2}
with the fact that an $r$-wise intersecting family is $2$-wise 
$(r-1)$-intersecting, see \cite{BL} for more details.
The case $r=2$ in the conjecture is known to be true as the Hilton--Milner 
theorem \cite{HM}.
The case $r=3$ is studied in \cite{T2023}, and a $k$-uniform 
version of Theorem~\ref{thm2} is obtained, from which it follows that the 
conjecture fails if $n$ and $k$ are sufficiently large and roughly 
$\frac12<\frac kn\leq\frac 23$. In this case 
$\FF_3(n,k,1,k-1)$ or $\FF_3(n,k,1,k)$ has size larger than 
$\FF_3(n,k,1,3)$ and $\FF_3(n,k,2,k-1)$ (see Theorem~4 in \cite{T2023}).
Balogh and Linz \cite{BL} verified that 
$\FF_3(11,7,1,7)$ is indeed a counterexample to the conjecture. 
They constructed the families $\FF_r(n,k,1,(r-1)i+1)$
($i\leq\lfloor\frac{k-1}{r-1}\rfloor$), and suggested that the largest
family of them could be a counterexample if $n\approx kr/(r-1)$.
Here we show that Conjecture~\ref{ON-V conj} fails if $r\geq 3$, 
and $n$ and $k$ are sufficiently large and $k/n$ is roughly in 
$(\frac{r-2}{r-1},\frac{r-1}r)$. More precisely we prove the following.
Let $M_r(n,k)$ denote the maximum size of a non-trivial $r$-wise 
intersecting family in $\binom{[n]}k$.
\begin{thm}\label{thm:counterexample}
Let $r\geq 3$.
For every $\epsilon>0$ and every $\delta>0$,
there exists $n_0\in\N$ such that for all integers $n$ and $k$ with
$n>n_0$ and $\frac{r-2}{r-1}+\epsilon<\frac kn<\frac{r-1}r-\epsilon$, 
we have
\[
 (1-\delta)\binom{n-1}{k-1}\leq M_r(n,k)<\binom{n-1}{k-1}.
\]
\end{thm}
Before proving this result, let us check that it gives counterexamples 
to the conjecture. To this end, suppose that $\frac kn=p$, and $n$ and $k$ 
are sufficiently large. Then we have 
\begin{align*}
&|\FF_r(n,k,r,1)|=(r+1)\binom{n-r-1}{k-r}+\binom{n-r-1}{k-r-1}
\approx\left((r+1)p^{r-1}q+p^r\right) \binom nk,\\
&|\FF_r(n,k,r-1,k-r+2)|=\binom{n-r+1}{k-r+1}-\binom{n-k-1}{k-r+1}+r-1
\approx p^{r-1} \binom nk.
\end{align*}
If $p>\frac 1r$ then $(r+1)p^{r-1}q+p^r>p^{r-1}$.
Indeed if $k>r$ and $n\leq r(k-r+2)$, then 
$|\FF_r(n,k,r,1)|>|\FF_r(n,k,r-1,k-r+2)|$.
If moreover $p=\frac kn\leq\frac{r-1}r$, then
\[
\lim_{n,k\to\infty} |\FF_r(n,k,r,1)|/\binom{n-1}{k-1}
=(r+1)p^rq+p^r\leq\frac 89,
\]
where equality holds if and only if $r=3$ and $p=\frac 23$.
This implies that under the assumptions in Theorem~\ref{thm:counterexample}
we have 
$\max\{|\FF_r(n,k,1,r)|,\,|\FF_r(n,k,r-1,k-r+2)|\}<(1-\delta)\binom{n-1}{k-1}$ for $0<\delta<\frac19$.

\begin{proof}[Proof of Theorem~\ref{thm:counterexample}]
The upper bound $M_r(n,k)<\binom{n-1}{k-1}$ was proved by Frankl in 
\cite{Fshift}.

We prove the lower bound.
Let $r$ be fixed, and let $\epsilon>0$ and $\delta>0$ be given.
Let $\frac{r-2}{r-1}<p<\frac{r-1}r$, and $k=pn$. 
Let $c>0$ be a constant depending on $r$ only (specified later), and let
\[
 J_{n,p}=\{j\in\N:|j-p^2n|\leq c\sqrt{n}\}.
\]
For $j\in J_{n,p}$ let
\[
 \theta_j(n,p)=\frac{\binom {pn}j\binom{n-pn}{pn-j}}{\binom n{pn}}
=\frac{\binom kj\binom{n-k}{k-j}}{\binom nk}.
\]
Let $\erf(z)$ denote the error function, that is,
\[
 \erf(z)=\frac 2{\sqrt\pi}\int_0^z\exp(-x^2)\,dx.
\]
Then, by Lemma~5 of \cite{T2023}, we have
\[
 \lim_{n\to\infty}\sum_{j\in J_{n,p}}\theta_j(n,p)=\erf\left(\frac{3c}{\sqrt{2}p}\right).
\]
The RHS is a function decreasing in $p$ (for fixed $c$), and we have
\[
\min_{p\in[\frac{r-2}{r-1},\frac{r-1}r]}
\erf\left(\frac{3c}{\sqrt 2 p}\right)
=\erf\left(\frac{3rc}{\sqrt 2(r-1)}\right).
\]
Then the RHS is a function increasing in $c$ and approaching $1$.
Thus we can choose $c>0$ so that
$\erf\left(\frac{3rc}{\sqrt 2(r-1)}\right)>1-\frac{\delta}3$,
and we fix $c$.

We will show that $|\FF_r(n,k,1,k)|>(1-\delta)\binom{n-1}{k-1}$.
Let $j_0=\lceil\frac{(r-2)k+1}{r-1}\rceil$.
Choose $n_1$ so that if $n>n_1$ then 
\begin{align}\label{eq:>1-delta}
 \sum_{j\in J_{n,p}}\theta_j(n,p)>1-\frac{\delta}2
\end{align}
for all $p$ with $\frac{r-2}{r-1}\leq p\leq\frac{r-1}r$. Next choose 
$n_2$ so that if $\frac{r-2}{r-1}+\epsilon<p<\frac{r-1}r-\epsilon$,
$n>n_2$, and $k=pn$, 
then $j_0<pk-c\sqrt n$ and $pk+c\sqrt n<k-1$. Then we have
$J_{n,p}\subset[j_0,k-1]$.
Finally choose $n_3$ so that if $n>n_3$ then 
$p-\frac c{q\sqrt n}>(1-\frac{\delta}2)p$, and 
let $n_0:=\max\{n_1,n_2,n_3\}$. 
 
We have
\begin{align*}
 |\FF_r(n,k,1,k)|
\geq\sum_{j=j_0}^{k-1}\binom kj\binom{n-k-1}{k-j-1}
>\sum_{j\in J_{n,p}}\binom kj\binom{n-k-1}{k-j-1}.
\end{align*}
The summands in the RHS is 
$\binom kj\binom{n-k-1}{k-j-1}=\frac{k-j}{n-k}\binom kj\binom{n-k}{k-j}$.
For $j\in J_{n,p}$ we have $j<p^2n+c\sqrt n$ and
\[
\frac{k-j}{n-k}=\frac{p-\frac jn}{1-p}
>\frac1q\left(p-p^2-\frac c{\sqrt n}\right)
=p-\frac{c}{q\sqrt n}
>\left(1-\frac{\delta}2\right)p, 
\]
where we used $n>n_3$ in the last inequality. We then have
\[
\binom kj\binom{n-k-1}{k-j-1}> 
\left(1-\frac{\delta}2\right)p\binom kj\binom{n-k}{k-j}.
\]
The RHS can be rewritten as 
$\left(1-\frac{\delta}2\right) \binom {n-1}{k-1}\theta_j(n,p) $
because $\binom kj\binom{n-k}{k-j}=\theta_j(n,p)\binom nk$ and
$p\binom nk=\binom{n-1}{k-1}$. Finally we have
\begin{align*}
M_r(n,k)&\geq |\FF_r(n,k,1,k)|\\
&>\sum_{j\in J_{n,p}}\binom kj\binom{n-k-1}{k-j-1}\\
&>\left(1-\frac{\delta}2\right) \binom {n-1}{k-1}\sum_{j\in J_{n,p}}
\theta_j(n,p) \\
&>\left(1-\frac{\delta}2\right)^2 \binom {n-1}{k-1} 
\qquad \text{(by \eqref{eq:>1-delta})}\\
&>\left(1-\delta\right) \binom {n-1}{k-1},
\end{align*}
and this is the lower bound we needed.
\end{proof}

Fact~\ref{fact 4-wise} suggests that the conjecture could be false even if
$\frac kn<\frac{r-2}{r-1}$. For example we have
$|\FF_4(41,26,2,25)|>|\FF_4(41,26,1,4)|>|\FF_4(41,26,3,24)|$.
Noting that $\frac{1+\sqrt{17}}8\approx 0.64$ we can expect
$\FF_4(n,k,2,k-1)$ is larger than $\FF_4(n,k,1,4)$ if
$\frac12<\frac kn<0.64$ and $n,k$ sufficiently large. Indeed we have
\begin{align*}
&|\FF_4(1000,514,2,513)|/|\FF_4(1000,514,1,4)|\approx 1.03254,\\
&|\FF_4(1000,630,2,629)|/|\FF_4(1000,630,1,4)|\approx 1.0165,\\
&|\FF_4(1000,650,2,649)|/|\FF_4(1000,650,1,4)|\approx 0.98655.
\end{align*}

\begin{prob}
Let $k>r\geq 2$ and $n\geq kr/(r-1)$, and  
let $\FF\subset\binom{[n]}k$ be a non-trivial $r$-wise intersecting family.
Is it true that 
\[
|\FF|\leq\max\{|\FF_r(n,k,s,y)|:1\leq s\leq r-1,\,r-s+1\leq y\leq k-s+1
\}\,?
\]
\end{prob}

\section*{Acknowledgment}
I thank both referees for their careful reading and many helpful 
suggestions. This research was supported by JSPS KAKENHI 
Grant No.~18K03399 and No.~23K03201.

\end{document}